\documentclass[reqno, 12pt]{article}

\usepackage{enumerate}
\usepackage{needspace}
\usepackage{latexsym}
\usepackage{amsmath}
\usepackage{amsfonts}
\usepackage{amsthm}
\usepackage{amssymb,mathtools}
\usepackage{newlfont}
\usepackage{graphics}
\usepackage{color}
\usepackage{float}
\usepackage{diagbox}
\usepackage{tocloft}
\usepackage{titlesec}
\usepackage{booktabs,array,longtable,enumitem,microtype}
\usepackage{extpfeil}
\usepackage{centernot}

\usepackage[linesnumbered,ruled,vlined]{algorithm2e}
\usepackage{url}

\usepackage[margin=1in]{geometry}

\usepackage[T1]{fontenc}
\usepackage{lmodern}
\usepackage{microtype}
\usepackage{rotating}
\usepackage{multirow}
\usepackage{extarrows}
\usepackage[sort,compress,numbers]{natbib}
\usepackage[utf8]{inputenc}
\usepackage{xcolor}
\usepackage{listings}
\usepackage{aliascnt}
\usepackage[pagebackref,colorlinks=true,linkcolor=blue,citecolor=red,urlcolor=blue]{hyperref}
\usepackage[nameinlink,noabbrev,capitalize]{cleveref}
\numberwithin{equation}{section}

\setlist[enumerate]{label=\textup{(\roman*)},leftmargin=*}
\newtheorem{theorem}{Theorem}[section]
\newaliascnt{proposition}{theorem}
\newtheorem{proposition}[proposition]{Proposition}
\aliascntresetthe{proposition}
\newaliascnt{lemma}{theorem}
\newtheorem{lemma}[lemma]{Lemma}
\aliascntresetthe{lemma}
\newaliascnt{corollary}{theorem}
\newtheorem{corollary}[corollary]{Corollary}
\aliascntresetthe{corollary}

\theoremstyle{definition}
\newaliascnt{definition}{theorem}

\aliascntresetthe{definition}
\newaliascnt{remark}{theorem}
\newtheorem{remark}[remark]{Remark}
\aliascntresetthe{remark}
\newaliascnt{example}{theorem}
\newtheorem{example}[example]{Example}
\aliascntresetthe{example}
\newaliascnt{conjecture}{theorem}

\aliascntresetthe{conjecture}
\newaliascnt{question}{theorem}

\aliascntresetthe{question}
\newaliascnt{problem}{theorem}
\newtheorem{problem}[problem]{Problem}
\aliascntresetthe{problem}
\newaliascnt{alg}{theorem}

\aliascntresetthe{alg}

\allowdisplaybreaks[4]

\SetKwInput{KwInput}{Input}                
\SetKwInput{KwOutput}{Output}              

\newcommand{\K}{\mathbb K}

\newcommand{\Ring}{R}

\title{Hankel Transform and $(\alpha,\beta)$ Somos-4 Sequences}

\author{Feihu Liu$^{\color{blue} \dag}$, Ying Wang$^{\color{blue} \ddag}$, and Zihao Zhang$^{\color{blue} \S}$
\\[2mm]
{\small $^{\color{blue} \dag}$ Center for Combinatorics, LPMC,}\\[-0.8ex]
{\small Nankai University, Tianjin 300071, P.R.~China}\\
{\small $^{\color{blue} \ddag}$ School of Mathematics and Statistics,}\\[-0.8ex]
{\small North China University of Water Resources and Electric Power, Zhengzhou, 450045, P.R.~China}\\
{\small $^{\color{blue} \S}$ School of Mathematics and Statistics,}\\[-0.8ex]
{\small Beijing Institute of Technology, Beijing 102400, P.R.~China}\\
{\small {\color{blue} $^\dag$} Email address: liufeihu7476@163.com}\\
{\small {\color{blue} $^\ddag$} Email address: wangying2019@ncwu.edu.cn}\\
{\small {\color{blue} $^\S$} Email address: zihao-zhang@foxmail.com}\\
}

\date{\today}

\begin{document}

\maketitle

\begin{abstract}
An $(\alpha,\beta)$ Somos-$4$ sequence $S_n$ is defined by the recurrence $S_nS_{n-4}=\alpha S_{n-1}S_{n-3}+\beta S_{n-2}^2$ ($n\geq 4$), with suitable initial values, where $\alpha$ and $\beta$ are constant parameters. A widely studied question is the following: When does the Hankel transform of a generating function become an $(\alpha,\beta)$ Somos-4 sequence? In particular, how can $\alpha$ and $\beta$ be derived for such a function? A sufficient condition for this problem has been established by Wang and Zhang. In this paper, we obtain the following three main results.

(i): We extend the Wang--Zhang sufficient condition by working over the rational function field. Then we combine this result with the Sulanke--Xin quadratic transformation to resolve all of Barry's currently unsolved $(\alpha,\beta)$ Somos-4 conjectures, which arise in diverse contexts, including generalized Catalan recurrences, Riordan arrays, generalized Bernstein arrays, and elliptic curves.

(ii): We show that the odd and even subsequences of an $(\alpha,\beta)$ Somos-4 sequence are again $(\alpha,\beta)$ Somos-4 sequences with transformed parameters. This is employed to establish Barry's Hurwitz transform conjecture.

(iii): Using the theory of orthogonal polynomials, we prove a Hankel determinant formula and thereby prove a conjecture related to the $(\alpha,\beta)$ Somos-4 sequence. In addition, we prove some conjectures on formulas for periodic Hankel determinants.
\end{abstract}

\noindent
\begin{small}
\end{small}

\noindent
\begin{small}
\end{small}

\tableofcontents

\section{Introduction}

Somos sequences were introduced by Michael Somos in 1989 as a curious generalisation of the recurrence that generates the numbers
$1,1,1,1,2,3,7,23,59,\dots$; see \cite{Somos89} or \cite[A006720]{Sloane23}. 
It is now called the \emph{Somos-4 sequence}.
The \emph{Somos‑$k$ sequence} $(S_n)_{n\ge 0}$ is defined by the initial
conditions $S_0 = S_1 = \dots = S_{k-1} = 1$,
and for all $n\ge k$,
$$S_nS_{n-k}=\sum_{j=1}^{\lfloor k/2\rfloor} S_{n-j}S_{n-(k-j)}.$$
The $k$ in the name is the order of the recurrence relation.
The most striking feature of these sequences is that, although they are defined by a rational recursion, for suitably chosen initial values they consist entirely of integers.

Somos-$k$ sequences have been extensively studied due to their connections with cluster algebras \cite{FZ02,FZ03}, elliptic curves \cite{Silverman2009,Hone2021}, and discrete integrable systems \cite{Hone05,ChangHuXin2015}. For further related results, we refer the reader to \cite{Hone07,BogdanovEtAl2023,Robinson92}.
This paper mainly concerns the connection between the Somos-4 sequence and Hankel determinants, in particular the $(\alpha,\beta)$ Somos-4 sequences.

An \emph{$(\alpha,\beta)$ Somos-4 sequence $S_n$} is a sequence such that 
\begin{align}\label{eq:somos4-bilinear}
S_nS_{n-4}=\alpha S_{n-1}S_{n-3}+\beta S_{n-2}^2 \qquad(n\geq4),
\end{align}
for appropriate initial values, where $\alpha$ and $\beta$ are constant parameters.
It has been studied in \cite{Hone05,Shipsey,Ward1948,WardDuck} and related works.
When $\alpha=\beta=1$ and $S_0=S_1=S_2=S_3=1$, we obtain the ordinary Somos-4 sequence.

We now introduce another object of study closely related to Somos-4: the Hankel transform.
Let \(F(x)=\sum_{n\geq0}f_nx^n\) be a formal power series. 
The \emph{Hankel determinant} of order $n$ of $F(x)$ is defined by 
\[ H_n(F)=\det(f_{i+j})_{0\leq i,j<n},\qquad H_0(F)=1.
\]
The sequence $(H_n)_{n\ge0}$ is called the \emph{Hankel transform} of the sequence $(f_n)_{n\ge0}$.

Hankel determinants have drawn significant attention from scholars in combinatorics and number theory.
For instance, Hankel determinants for specific sequences have been explored in several works \cite{chien2022hankel,Elouafi,Mu-Wang,Mu-Wang-Yeh,QH-Hou,Wang-Xin,ChangHu2013,ChangLiu2026,HanTAMS}, among others.
Various continued fraction methods for have been developed, including:
S-continued fraction (or J-continued fraction) \cite{Krattenthaler05, FlajoletDM}, C-continued fraction \cite{Cigler-C13,PaulBarry}, Gessel--Xin's continued fraction method \cite{Gessel-Xin}, Sulanke-Xin's continued fraction method \cite{SulankeXin2008}, the Hankel continued fraction \cite{HanGuoNiu}, and the shifted periodic continued fraction \cite{Wang-Xin-Zhai}.

The following problem, involving Hankel transforms and $(\alpha,\beta)$ Somos-4 sequences, has attracted considerable attention.
\begin{problem}\label{Openproblem}
When does the Hankel transform of a generating function $F(x)$ become an $(\alpha,\beta)$ Somos-4 sequence?
In particular, how can $\alpha$ and $\beta$ be derived for such a function?
\end{problem}

Regarding \cref{Openproblem}, Barry observed many instances in which Hankel transforms of quadratic algebraic generating functions satisfy \eqref{eq:somos4-bilinear}; see \cite{Barry2010,Barry2011Invariant,Barry2011Narayana,Barry2012Hurwitz,Barry2012Bernstein,
Barry2019AMatrix,Barry2019Pseudo,Barry2021Catalan,Barry2022Conjectures,Barry2023Elliptic,Barry2023Motzkin}.
Several of these observations were stated as parameterized conjectures.
Rajkovi\'c, Barry, and Savi\'c \cite{RajkovicBarrySavic2012} also proposed a conjecture on the Somos-4 sequence.

Some of the many conjectures in the literatures cited above have been resolved. 
For example, 
Chang and Hu~\cite{ChangHu2012} proved the convolution family from \cite{Barry2010} by block-Hankel identities.
Xin~\cite{Xin2009} also proved the famous Somos-4 conjecture proposed by Somos.
Wang and Zhang~\cite{WangZhang2024} proved the four $(\alpha,\beta)$ Somos-4 conjectures of \cite{Barry2022Conjectures} by
extracting a sufficient condition from Xin's recursion system \cite{Xin2009}.

This sufficient condition is now referred to as the \emph{Wang--Zhang condition}, see \cref{thm:wz}.
However, the Wang--Zhang sufficiency theorem requires that some nonzero conditions hold.
Inspired by the proof of the Wang--Zhang theorem, our \textbf{first} contribution is to prove an extension of this theorem over the field of rational functions. This broadens the scope of the Wang--Zhang conditions.

\begin{theorem}\label{prop:wz-universal}
Let \(\mathcal A\) be a commutative ring with identity, and let \(a_0,b_0,c_0,d_0,f_0\in\mathcal A\). There is a unique series
\(Q(x)\in\mathcal A[[x]]\) satisfying
\begin{equation}\label{eq:wz-cross-multiplied}
 (1+c_0x+d_0x^2)Q(x)+x^2(-1+f_0x)Q(x)^2=a_0+b_0x.
\end{equation}
Equivalently, \(Q(x)\) satisfies \eqref{eq:wz-form}.  Set $\Delta=a_0(c_0+f_0)-b_0$, $\Gamma=a_0^3-a_0^2d_0+a_0b_0c_0-b_0^2$,
and
\begin{equation}\label{eq:wz-polynomial-parameters}
 \alpha=\Delta^2,\qquad\beta=(a_0-d_0-c_0f_0-f_0^2)\Gamma-a_0\Delta^2.
\end{equation}
Then
\begin{equation}\label{eq:wz-universal-somos}
 H_n(Q)H_{n-4}(Q) =\alpha H_{n-1}(Q)H_{n-3}(Q) +\beta H_{n-2}(Q)^2\qquad(n\geq4).
\end{equation}
If \(a_0\) is a unit in \(\mathcal A\) (in particular, if \(\mathcal A\) is a field and \(a_0\ne0\)), then the definitions of \(f_1\) and \(a_1\) in \cref{thm:wz} make sense, and \eqref{eq:wz-polynomial-parameters} agrees with \eqref{eq:wz-alpha}--\eqref{eq:wz-beta}.
\end{theorem}

\cref{prop:wz-universal} provides a partial answer to \cref{Openproblem}.
In this paper, \cref{prop:wz-universal} will be fully exploited to investigate the many $(\alpha,\beta)$ Somos‑4 conjectures mentioned above.

Motivated by the Hurwitz transform \cite{Barry2012Hurwitz}, we prove the following theorem.
This is the \textbf{second} main contribution of this paper.
It shows that the even and odd subsequences of an $(\alpha,\beta)$ Somos-4 sequence also form a Somos-4 sequence.
\begin{theorem}\label{lem:parity-decimation}
Let \((S_n)_{n\geq 0}\) be an $(\alpha,\beta)$ Somos-4 sequence such that $S_n$ satisfies
\begin{align*}
S_{n+2}S_{n-2}=\alpha S_{n+1}S_{n-1}+\beta S_n^2\quad \text{ for }\quad n\geq 2.
\end{align*}
Thus the terms occurring in the quotient variables below are assumed to be nonzero. Set
\[x_n=\frac{S_{n+1}S_{n-1}}{S_n^2}\quad\text{for}\quad n\geq 1,\qquad
 J=x_{n-1}x_n+\alpha\left(\frac1{x_{n-1}}+\frac1{x_n}\right)+\frac{\beta}{x_{n-1}x_n},
\]
and \(\Lambda=\alpha^2+\beta J\).  Then, for \(\varepsilon\in\{0,1\}\), the sequence \((S_{2m+\varepsilon})_{m\geq 0}\) is a
\[\left(\Lambda^2,\,\beta\bigl(\beta^3-\Lambda(\Lambda+\alpha^2)\bigr)\right)
\]
Somos-4 sequence. 
\end{theorem}

Applying \cref{lem:parity-decimation}, we establish the general form of Barry's conjecture \cite[Example~7]{Barry2012Hurwitz}.

The \textbf{third} main contribution of this paper is to prove several Hankel determinant conjectures, which are closely related to the $(\alpha,\beta)$ Somos-4 sequence. This contribution consists of two main parts.
The first is a Hankel determinant formula derived from orthogonal polynomial theory (see \cref{lem:constant-tail-border}). This result is then employed to prove the conjectures in~\cite{Barry2021Catalan}.
The second is to prove several conjectured formulas for periodic Hankel determinants (see \cref{thm:2021-c39,thm:new-a136576}).

The \textbf{fourth} contribution of this paper is to settle many conjectures on $(\alpha,\beta)$ Somos‑4 sequences related to Hankel determinants. They appear respectively in various forms and backgrounds related to Riordan arrays, Bernstein arrays, A-matrices, elliptic-curve families, and so on.
In particular, the authors summarize the main results of this paper and various previously proved results concerning the Somos‑4 sequence conjecture, and present them in the list below.

\begin{longtable}{@{}
 >{\raggedright\arraybackslash}p{0.10\textwidth}
 >{\raggedright\arraybackslash}p{0.55\textwidth}
 >{\raggedright\arraybackslash}p{0.30\textwidth}@{}}
\toprule
Year & Source assertion & Proved\\
\midrule
\endhead
2000 & Somos, \cite{Somos2000}. & Xin, \cite{Xin2009}.\\

2007 & Gosper and Schroeppel, \cite[Conjecture 4; Conjecture 4.5]{Gosper2007}. & Ma, \cite{Ma2010}.\\

2008 & OEIS, \cite[A136576; A136577]{Sloane23} (Determinant conjecture). & \cref{thm:new-a136576}.\\

2009 & OEIS, \cite[A160702; A160703]{Sloane23}. & \cref{thm:new-a160702}; \cref{cor:new-a160703}.\\

2010 & Barry, \cite[Conjecture~10]{Barry2010}. & Chang--Hu, \cite{ChangHu2012}; \cref{cor:barry2010}.\\

2011 & Barry, \cite[Conjecture~26]{Barry2011Narayana}. & \cref{thm:cubic-master}.\\

2011 & Barry, \cite[Conjecture~5]{Barry2011Invariant}. & \cref{Corollary-Conjecture-5-Bary}.\\

2012 & Barry, \cite[Example~7]{Barry2012Hurwitz}. & \cref{thm:new-hurwitz}; \cref{Corollary-Hurwitz}.\\

2012 & Rajkovi\'c--Barry--Savi\'c, \cite[Section~7]{RajkovicBarrySavic2012}. & \cref{Corollary-RBS-111}.\\

2012 & Barry, \cite[Page~11; Page~12; Page~13]{Barry2012Bernstein} and \cite[Conjecture 15; Conjecture 16; Conjecture 17]{Barry2012Bernstein}. & \cref{thm:bernstein-catalan}; \cref{thm:bernstein-c16}; \cref{cor:bernstein-c15}; \cref{thm:bernstein-c171}; \cref{thm:bernstein-c172}; \cref{thm:bernstein-c173}.\\

2019 & Barry, \cite[Conjecture~6; Conjecture~14]{Barry2019AMatrix}. & \cref{thm:amatrix-c6}; \cref{thm:amatrix-c14}.\\

2019 & Barry, \cite[Page~11; Page~17]{Barry2019Pseudo}. & \cref{thm:pseudo-tail}; \cref{thm:pseudo-asequence}.\\

2021 & Barry, \cite[Conjecture~14; Conjecture~15]{Barry2021Catalan}. & \cref{thm:2021-c14}; \cref{thm:2021-c15}.\\

2021 & Barry, \cite[Conjectures~24; Conjecture~25]{Barry2021Catalan}. & \cref{thm:2021-c24}; \cref{thm:2021-c25}.\\

2021 & Barry, \cite[Conjecture~26; Conjecture~30; Conjecture~39]{Barry2021Catalan} (Determinant conjecture). & \cref{thm:2021-c26-c30}; \cref{thm:2021-c39}.\\

2021 & Barry, \cite[Page~27]{Barry2021Catalan}. & \cref{cor:2021-p27}.\\

2021 & Barry, \cite[Conjecture~43; Conjecture~44]{Barry2021Catalan}. & \cref{thm:2021-c43}; \cref{thm:2021-c44}.\\

2021 & Barry, \cite[Page~40]{Barry2021Catalan}. & \cref{thm:2021-p39}.\\

2022 & Barry, \cite[Page~3]{Barry2022Conjectures}. & \cref{thm:2022-unnumbered}.\\

2022 & Barry, \cite[Conjecture~2; Conjecture~3; Conjecture~4; Conjecture~5]{Barry2022Conjectures}. & Wang--Zhang, \cite{WangZhang2024}.\\

2023 & Barry, \cite[Conjectures~2]{Barry2023Elliptic}. & \cref{thm:elliptic-2023}; \cref{thm:elliptic-Division-2023}.\\

2023 & Barry, \cite[Conjecture~8; Example 14]{Barry2023Motzkin}. & \cref{Conjecture-8-Barry-ARST}; \cref{thm:new-motzkin-9-5}.\\

2025 &  Ovsienko--Pedon, \cite[Conjecture 1.9]{OvsienkoPedon}. & Han--Pedon, \cite{Han-Pedon}.\\

\bottomrule
\end{longtable}

To the best of our knowledge, we have perhaps proved all of Barry's unsolved $(\alpha,\beta)$ Somos-4 conjectures.

The paper is organized as follows.
In \cref{Preliminaries}, we give some lemmas that will be used later. In particular, we shall extend the sufficient conditions proposed by Wang--Zhang~\cite{WangZhang2024}.
In \cref{sec:new-hurwitz}, we prove \cref{lem:parity-decimation} and the Hurwitz transform conjecture.
\cref{sec:catalan-2021} is mainly devoted to proving a determinant formula (\cref{lem:constant-tail-border}) and the $(\alpha,\beta)$ Somos-4 conjecture associated with the Catalan recurrence.
In \cref{Section-Elliptic-PF,Section-Tree,sec:bernstein,Section-Five}, we mainly prove the theorems and corollaries listed in the above table by using \cref{prop:wz-universal} and Sulanke-Xin quadratic transformation $\tau$ \cite[Proposition~4.1]{SulankeXin2008}.
\cref{Section-Seven-Hankel-Evalu} mainly resolves some conjectures on expressions involving periodic Hankel determinants.
\cref{sec:new-motzkin} proves a conjecture on generating functions as a byproduct.

\section{A sufficient condition}\label{Preliminaries}

This section mainly introduces some lemmas and previously known results.
These results will be used repeatedly in later sections.
For completeness, we sometimes also provide simple proofs for some of the known results.

\subsection{Preliminaries}

We now recall the \emph{Catalan generating function} $\mathcal{C}(x)$. It is defined as
$$\mathcal{C}(x)=\frac{1-\sqrt{1-4x}}{2x}.$$
The function $\mathcal{C}(x)$ satisfies \(\mathcal{C}(x)=1+x\mathcal{C}(x)^2\).
We know that $\mathcal{C}(x)$ can be expressed as the continued fraction 
\[\mathcal{C}(x)=\cfrac1{1-
       \cfrac{x}{1-
       \cfrac{x}{1-\ddots}}}.
\]
The Catalan generating function will be used frequently.

We work over a field \(\K\) of characteristic zero.  Every algebraic equation below is understood to select the unique formal power series with the stated constant term.

\begin{lemma}\label{lem:scalings}
Let \(F(x)=\sum_{n\geq0}f_nx^n\) be a formal power series, and let \(u,v\in\K\).  Then
\begin{align*}
 H_n(uF(x))=u^nH_n(F(x)),\qquad
 H_n(F(vx))=v^{n(n-1)}H_n(F(x)).
\end{align*}
\end{lemma}
\begin{proof}
The first identity factors \(u\) from every row.  For the second, factor \(v^i\) from row \(i\) and \(v^j\) from column \(j\).
\end{proof}

In the study of orthogonal polynomials and algebraic combinatorics, the bivariate Hankel kernel provides a powerful generating function approach to evaluating and manipulating Hankel determinants. 
Let $A(x) = \sum_{n=0}^{\infty} a_n x^n \in \mathbb{C}[[x]]$ be a formal power series.

The \emph{bivariate Hankel kernel} associated with the formal power series $A(x)$ is defined as the formal bivariate rational series
\begin{align*}
\mathcal{K}_A(x,y) = \frac{xA(x) - yA(y)}{x - y}.
\end{align*}
Utilizing the standard algebraic identity for the difference of powers, $x^{n+1} - y^{n+1} = (x-y) \sum_{i+j=n} x^i y^j$, one obtains the formal expansion:
\begin{align*}
\mathcal{K}_A(x,y) = \frac{1}{x-y} \sum_{n=0}^{\infty} a_n (x^{n+1} - y^{n+1}) = \sum_{n=0}^{\infty} a_n \sum_{i+j=n} x^i y^j = \sum_{i,j \ge 0} a_{i+j} x^i y^j.
\end{align*}
Consequently, the coefficient of $x^i y^j$ in the series $\mathcal{K}_A(x,y)$ is precisely $a_{i+j}$. Thus, the infinite coefficient matrix of $\mathcal{K}_A(x,y)$ corresponds exactly to the Hankel matrix $(a_{i+j})_{i,j \ge 0}$ associated with the sequence $(a_n)_{n \ge 0}$.
A nice application related to the Hankel kernel can be found in \cite{Gessel-Xin}.

A fundamental application of the bivariate Hankel kernel arises in the evaluation of Hankel determinants via generating function transformations, specifically through the application of unitriangular congruences.

\begin{proposition} \label{prop:unit_congruence}
Let $U(x) \in 1 + x\mathbb{C}[[x]]$ be a unit formal power series. The transformation
\begin{align*}
    \widetilde{\mathcal{K}}(x,y) = U(x) \mathcal{K}_A(x,y) U(y)
\end{align*}
corresponds to a unitriangular congruence of the associated Hankel matrices. In particular, this operation preserves all leading principal minors, thereby leaving the Hankel determinants invariant.
\end{proposition}
\begin{proof}[Sketch of the proof]
Let $H_A$ denote the coefficient matrix of $\mathcal{K}_A(x,y)$. The multiplication of the kernel by $U(x)U(y)$ translates, at the matrix level, to the congruence transformation $\widetilde{H} = L H_A L^T$, where $L$ is the lower triangular Toeplitz matrix generated by the coefficients of $U(x)$. Because the constant term of $U(x)$ is $1$, the matrix $L$ is unit lower triangular (its main diagonal consists entirely of ones). By the Cauchy--Binet formula, such a unitriangular congruence preserves all leading principal minors of $H_A$.
\end{proof}

\begin{lemma}\label{lem:jacobi-step}
If the power series $F(x)$ and $G(x)$ satisfy
\[F(x)=\frac{1}{1-ux-vx^2G(x)},
\]
then
\[H_n(F)=v^{\,n-1}H_{n-1}(G)\qquad(n\geq1).
\]
\end{lemma}
\begin{proof}
Let \(Q(x)=vG(x)\). The bivariate generating kernel of the Hankel matrix of a series \(A(x)\) is
\[\mathcal K_A(x,y)=\frac{xA(x)-yA(y)}{x-y}=\sum_{i,j\geq0}a_{i+j}x^iy^j.
\]
Left and right multiplication by the unit lower-triangular Toeplitz matrix associated with \(1/F(x)\) changes this kernel to
\begin{align*}
 \frac{\mathcal K_F(x,y)}{F(x)F(y)}=\frac{x/F(y)-y/F(x)}{x-y}=1+xy\frac{xQ(x)-yQ(y)}{x-y}.
\end{align*}
Thus every finite principal block is transformed by determinant-one row and column operations into \(1\oplus H_{n-1}(Q)\).  Hence
\(H_n(F)=H_{n-1}(Q)=v^{n-1}H_{n-1}(G)\).
\end{proof}

The following lemma is a classical result concerning Hankel determinants of J-fractions; we refer the reader to \cite{Krattenthaler1999,Krattenthaler05,Viennot,Wall2000} for details.

\begin{lemma}\label{lem:jacobi-product}
If a normalized series has Jacobi continued fraction
\begin{align}\label{Equation-FX-230-UV}
F(x)=\sum_{k\geq 0}\mu_k x^k=\cfrac{\mu_0}{1-u_0x-
       \cfrac{v_1x^2}{1-u_1x-
       \cfrac{v_2x^2}{1-u_2x-\ddots}}},
\end{align}
then we have 
\begin{align*}
 H_n(F)=\mu_0^n\prod_{j=1}^{n-1}v_j^{\,n-j}\qquad(n\geq1).
\end{align*}
\end{lemma}

Sulanke and Xin \cite[Proposition~4.1]{SulankeXin2008} define a quadratic transformation $\tau$ on a certain generating function $F(x)$ of algebraic degree $2$. Then they established a direct and simple connection between $H_n(F)$ and $H_n(\tau(F))$, namely, $H_n(F)=a^nH_{n-d-1}(\tau(F))$, where $a$ is a constant and $d\geq 0$.
Applying the transformation $\tau$, Xin obtained the following result.

\begin{lemma}{\em (\cite[Lemma 1]{Xin2009})}\label{lem:sx}
Suppose \(a\ne0\) and
\[F(x)=\frac{a+bx}{1+cx+dx^2+x^2(e+fx)F(x)}.
\]
Set
\begin{align*}
 A&=-\frac{a^3e+a^2d-abc+b^2}{a^2},\\
 B&=-\frac{a^4f+ca^3d-a^2c^2b+2acb^2-a^2bd-b^3}{a^3},\\
 D&=-\frac{a^2d-2abc+2b^2}{a^2},
\end{align*}
and define
\[G(x)=\frac{A+Bx}{1+cx+Dx^2+x^2(-1-(b/a)x)G(x)}.
\]
Then we have
\[H_n(F)=a^nH_{n-1}(G)\qquad(n\geq1).
\]
\end{lemma}

\subsection{Extension of the Wang--Zhang sufficient condition}

Using \cref{lem:sx}, Wang and Zhang proved the following result, which also provides a sufficient condition for \cref{Openproblem}.

\begin{theorem}{\em (\cite[Theorem~1]{WangZhang2024})}\label{thm:wz}
Suppose \(a_0\ne0\) and
\begin{equation}\label{eq:wz-form}
 Q(x)=\frac{a_0+b_0x}{1+c_0x+d_0x^2+x^2(-1+f_0x)Q(x)}.
\end{equation}
Set
\[f_1=-\frac{b_0}{a_0},\qquad
 a_1=a_0-d_0+\frac{b_0}{a_0}\left(c_0-\frac{b_0}{a_0}\right).
\]
Then \((H_n(Q))_{n\geq0}\) is a $(\alpha,\beta)$ Somos-4 sequence with
\begin{align}
 \alpha&=a_0^2(c_0+f_0+f_1)^2,\label{eq:wz-alpha}\\
 \beta&=-(c_0+f_0+f_1)^2a_0^3-a_1\bigl((f_0-f_1)(c_0+f_0+f_1)-a_1\bigr)a_0^2.
 \label{eq:wz-beta}
\end{align}
\end{theorem}

The sufficient condition in \cref{thm:wz}  is now referred to as the \emph{Wang--Zhang condition}.
According to the Wang--Zhang condition, Barry's four $(\alpha,\beta)$ Somos-4 conjectures \cite{Barry2022Conjectures} were proved directly in \cite[Corollaries~4--7]{WangZhang2024}.

\begin{remark}\label{rem:specialization}
We need to call the reader's attention to the following fact.
In \cref{lem:sx,thm:wz}, all parameters except $a$ and $a_0$ are allowed to be zero. 
In other words, $a$ and $a_0$ are not allowed to be zero.
In fact, the proof of \cref{thm:wz} requires all determinants $(H_n(Q))_{n\geq0}$ to be nonzero.
This is due to the fact that every subsequent application of \cref{lem:sx} requires the corresponding leading coefficient to be
nonzero.

We find that when the proof of \cref{thm:wz} is carried over to a field of fractions (the rational function field), these nonzero conditions are unnecessary.
Therefore, \cref{prop:wz-universal} below extends \cref{thm:wz}.
For convenience, we provide a complete proof.
In the subsequent sections, we shall frequently use \cref{lem:sx,thm:wz} to prove many conjectures on the $(\alpha,\beta)$ Somos-4 sequence. 
Concerning the vanishing of the nonzero condition, see the proof of \cref{prop:wz-universal} for details, so we do not repeat it.
\end{remark}

\begin{proof}[Proof of \cref{prop:wz-universal}]
We prove this theorem in the following five steps.

\smallskip\noindent
\emph{Step 1: The universal series $Q_0(x)$.}
We regard five symbols \(A,B,C,D,F\) as algebraically independent.
Let $\mathcal R$ be an integral domain and let $\mathcal K$ be its field of fractions:
$\mathcal R=\mathbb Z[A,B,C,D,F]$ and $\mathcal K=\operatorname{Frac}(\mathcal R)$.
Let \(Q_0(x)=\sum_{m\geq0}q_mx^m\) satisfy
\begin{align*}
 (1+Cx+Dx^2)Q_0(x)+x^2(-1+Fx)Q_0(x)^2=A+Bx.
\end{align*}
Set \(q_j=0\) for \(j<0\), and interpret an empty sum as zero.
By equating coefficients of \(x^m\), we have
\begin{align}\label{eq:wz-coefficient-recursion}
q_m=A\delta_{m,0}+B\delta_{m,1}-Cq_{m-1}-Dq_{m-2}+\sum_{\substack{r,s\geq0\\r+s=m-2}}q_rq_s
-F\sum_{\substack{r,s\geq0\\r+s=m-3}}q_rq_s,
\qquad m\geq0,
\end{align}
where \(\delta_{i,j}\) is the Kronecker delta.  Every term on the right-hand side of \eqref{eq:wz-coefficient-recursion}, apart from its constant terms, involves only coefficients preceding \(q_m\).
Consequently the recursion proves existence and uniqueness and, by induction, $q_m\in\mathcal R$, $H_m(Q_0)\in\mathcal{R}$ for $m\geq0$.
In particular, this construction does not require \(A\ne0\).
The first three coefficients are $q_0=A$, $q_1=B-AC$, and $q_2=A^2-AD-BC+AC^2$.
It follows that $H_2(Q_0)=q_0q_2-q_1^2=A^3-A^2D+ABC-B^2$.

\smallskip\noindent
\emph{Step 2: Apply \cref{lem:sx}.}
We now work over \(\mathcal K\).  Starting from $A_0=A$, $B_0=B$, $D_0=D$, $F_0=F$, consider the recursions
\begin{equation}\label{eq:wz-xin-orbit}
\begin{aligned}
 A_{j+1}&=-\,\frac{-A_j^3+A_j^2D_j-A_jB_jC+B_j^2}{A_j^2},\\
 B_{j+1}&=-\,\frac{A_j^4F_j+CA_j^3D_j-C^2A_j^2B_j +2CA_jB_j^2-A_j^2B_jD_j-B_j^3}{A_j^3},\\
 D_{j+1}&=-\,\frac{A_j^2D_j-2A_jB_jC+2B_j^2}{A_j^2},\\
 F_{j+1}&=-\,\frac{B_j}{A_j}.
\end{aligned}
\end{equation}
These are precisely the parameter recursions obtained from \cref{lem:sx} after setting \(e_j=-1\) and keeping \(c_j=C\). 
Whenever the quantities in \eqref{eq:wz-xin-orbit} are defined, let \(Q_j(x)\) be the unique series satisfying
\begin{align*}
 Q_j(x)=\frac{A_j+B_jx}{1+Cx+D_jx^2+x^2(-1+F_jx)Q_j(x)}.
\end{align*}
We must prove, rather than assume, that every \(A_j\) is a nonzero element of \(\mathcal K\). 
Let the prime ideal
\[\mathfrak m=(A-1,B,C,D,F)\subset\mathcal R
\]
and work momentarily in the local ring \(\mathcal R_{\mathfrak m}\).
We claim inductively that the four rational functions
\(A_j,B_j,D_j,F_j\) belong to \(\mathcal R_{\mathfrak m}\) and have
values
\begin{equation}\label{eq:wz-stationary-values}
 A_j(1,0,0,0,0)=1,\qquad
 B_j(1,0,0,0,0)=D_j(1,0,0,0,0)=F_j(1,0,0,0,0)=0.
\end{equation}
This is true for \(j=0\).  If it holds for \(j\), then the image of \(A_j\) in the residue field of \(\mathcal R_{\mathfrak m}\) is \(1\); hence \(A_j\) is a unit of \(\mathcal R_{\mathfrak m}\).
Thus every expression on the right-hand side of \eqref{eq:wz-xin-orbit} belongs to \(\mathcal R_{\mathfrak m}\).
Substituting the four values in \eqref{eq:wz-stationary-values}, together with \(C=0\), into those expressions gives respectively \(1,0,0,0\).  This proves the induction.
In particular, \(A_j\) takes the value \(1\) at this specialization and therefore cannot be the zero rational function in \(\mathcal K\).
All divisions in \eqref{eq:wz-xin-orbit} are consequently legitimate over \(\mathcal K\).

We may now apply \cref{lem:sx} at every step: $H_m(Q_j)=A_j^mH_{m-1}(Q_{j+1})$ for $m\geq1$.
Iterating this identity gives
\begin{equation}\label{eq:wz-Hankel-product}
 \mathcal H_m:=H_m(Q_0)=\prod_{j=0}^{m-1}A_j^{\,m-j}=A_0^mA_1^{m-1}\cdots A_{m-2}^2A_{m-1}.
\end{equation}
Hence every \(\mathcal H_m\) is nonzero in \(\mathcal K\). 
For each fixed \(m\), $\mathcal H_m$ uses only finitely many transformations.

\smallskip\noindent
\emph{Step 3: The recurrence associated with $A_j$.}
The last equation of \eqref{eq:wz-xin-orbit} gives $B_j=-A_jF_{j+1}$.
By substituting this relation in the first equation of \eqref{eq:wz-xin-orbit}, we obtain
\begin{equation}\label{eq:wz-D-expression}
 D_j=A_j-A_{j+1}-CF_{j+1}-F_{j+1}^2.
\end{equation}
Because \(B_{j+1}=-A_{j+1}F_{j+2}\), substitution of \(B_j=-A_jF_{j+1}\) into the second equation of
\eqref{eq:wz-xin-orbit} gives, term by term,
\begin{align*}
A_{j+1}F_{j+2}={}&A_jF_j+CD_j+C^2F_{j+1}+2CF_{j+1}^2+D_jF_{j+1}+F_{j+1}^3 \\
={}&A_jF_j+(C+F_{j+1})\bigl(D_j+CF_{j+1}+F_{j+1}^2\bigr).
\end{align*}
By \eqref{eq:wz-D-expression}, the expression in parentheses is \(A_j-A_{j+1}\). Therefore, we have 
\[A_{j+1}F_{j+2}=A_jF_j+(C+F_{j+1})(A_j-A_{j+1}),
\]
and hence
\begin{equation}\label{eq:wz-first-conservation}
A_{j+1}(F_{j+2}+F_{j+1}+C)=A_j(F_{j+1}+F_j+C)=A_0(F_1+F_0+C).
\end{equation}
Write $\sigma=F_0+F_1+C$, and $L=A_0\sigma$.
Then the common value in \eqref{eq:wz-first-conservation} is \(L\).

The third equation of \eqref{eq:wz-xin-orbit}, together with $B_j=-A_jF_{j+1}$, becomes
\begin{align*}
 D_{j+1}=-D_j-2CF_{j+1}-2F_{j+1}^2.
\end{align*}
By substituting \eqref{eq:wz-D-expression} into the above equation, we get
\begin{align*}
A_j-A_{j+2}=C(F_{j+2}-F_{j+1})+F_{j+2}^2-F_{j+1}^2=(F_{j+2}-F_{j+1})(F_{j+2}+F_{j+1}+C).
\end{align*}
By multiplying the above equation by \(A_{j+1}\) and using \eqref{eq:wz-first-conservation}, we obtain
\begin{equation}\label{eq:wz-telescoping-summand}
 A_{j+1}(A_j-A_{j+2}) =L(F_{j+2}-F_{j+1}).
\end{equation}
Summing \eqref{eq:wz-telescoping-summand} from \(j=0\) to \(j=r\) and telescoping on both sides gives 
\begin{equation}\label{eq:wz-second-conservation}
 A_0A_1-A_{r+1}A_{r+2}=L(F_{r+2}-F_1).
\end{equation}
We shall also use \eqref{eq:wz-second-conservation} for \(r=-1\), when
both sides are zero.
Define
\begin{equation}\label{eq:wz-theta-alpha}
 \Theta=2A_0A_1+L(C+2F_1),\qquad\alpha=L^2=A_0^2\sigma^2.
\end{equation}
According to \eqref{eq:wz-second-conservation}, we have
\[A_{j+1}(A_j+A_{j+2})=2A_0A_1-L(F_{j+1}+F_{j+2}-2F_1).
\]
By multiplying by \(A_{j+1}\) and using $A_{j+1}(F_{j+1}+F_{j+2}+C)=L$ from \eqref{eq:wz-first-conservation}, we obtain the recurrence
\begin{equation}\label{eq:wz-centered-recurrence}
 A_{j+1}^2(A_{j+2}+A_j)-\Theta A_{j+1}+\alpha=0 \qquad(j\geq0).
\end{equation}

\smallskip\noindent
\emph{Step 4: The Hankel identity.}
For \(j\geq0\), set $\mathcal I_j=\Theta A_jA_{j+1}-A_j^2A_{j+1}^2-\alpha(A_j+A_{j+1})$.
For \(j\geq1\), we have
\begin{align*}
\mathcal I_j-\mathcal I_{j-1}=(A_{j+1}-A_{j-1})\bigl(\Theta A_j-A_j^2(A_{j+1}+A_{j-1})-\alpha\bigr)=0,
\end{align*}
where the final equality is \eqref{eq:wz-centered-recurrence} with its index lowered by one. 
Thus \(\mathcal I_j=\mathcal I_0\) for all \(j\geq0\).
Using \eqref{eq:wz-theta-alpha} and \(L=A_0\sigma\), we find
\begin{align}
\mathcal I_0&=\Theta A_0A_1-A_0^2A_1^2-\alpha(A_0+A_1)=-A_0^3\sigma^2+A_0^2A_1\bigl(A_1+(F_1-F_0)\sigma\bigr)\notag\\
&=-A_0^3\sigma^2-A_0^2A_1\bigl((F_0-F_1)\sigma-A_1\bigr).\label{eq:wz-beta-generic}
\end{align}
We now denote this common value by \(\beta\), i.e., \(\mathcal I_j=\mathcal I_0=\beta\) for all \(j\geq0\). 
According to \eqref{eq:wz-centered-recurrence}, we have
$$A_{j-1}(A_j^2(A_{j+1}+A_{j-1})-\Theta A_j+\alpha)=0.$$
Furthermore, we obtain 
\begin{align}
 A_{j+1}A_j^2A_{j-1}&=\Theta A_jA_{j-1}-A_j^2A_{j-1}^2-\alpha A_{j-1}\notag\\
 &=\mathcal I_{j-1}+\alpha A_j=\alpha A_j+\beta\qquad(j\geq1).\label{eq:wz-A-somos}
\end{align}
Taking \(j=n-2\) in \eqref{eq:wz-A-somos}, we have $A_{n-1}A_{n-2}^2A_{n-3}=\alpha A_{n-2}+\beta$.
According to \eqref{eq:wz-Hankel-product}, we obtain, for \(n\geq4\),
\begin{align*}
 \mathcal H_n\mathcal H_{n-4}=\mathcal H_{n-2}^2A_{n-1}A_{n-2}^2A_{n-3}\quad \text{and}\quad
 \mathcal H_{n-1}\mathcal H_{n-3}=\mathcal H_{n-2}^2A_{n-2}.
\end{align*}
Therefore, we obtain $\mathcal H_n\mathcal H_{n-4}=\alpha\mathcal H_{n-1}\mathcal H_{n-3}+\beta\mathcal H_{n-2}^2$ for $n\geq4$ in $\mathcal K$.

\smallskip\noindent
\emph{Step 5: Removal of all nonvanishing assumptions.}
We first rewrite the two parameters as elements of \(\mathcal R\).  From \eqref{eq:wz-xin-orbit}, we know that 
\[F_1=-\frac BA,\qquad A_1=A-D+\frac BA\left(C-\frac BA\right).
\]
Consequently, with $\Delta=A(C+F)-B$, $\Gamma=A^3-A^2D+ABC-B^2$, we have
\begin{align*}
\sigma=\frac{\Delta}{A},\qquad L=\Delta,\qquad A_1=\frac{\Gamma}{A^2}.
\end{align*}
In particular, \(\alpha=L^2=\Delta^2\).  Furthermore, we have
\begin{align*}
 A_1-(F-F_1)\sigma &=A-D+\frac{B}{A}\left(C-\frac{B}{A}\right)-\left(F+\frac{B}{A}\right)\left(C+F-\frac{B}{A}\right)
 \\ &=A-D-CF-F^2.
\end{align*}
It follows from \eqref{eq:wz-beta-generic} that
\begin{align*}
\beta=-A^3\sigma^2+A^2A_1\bigl(A_1-(F-F_1)\sigma\bigr)=-A\Delta^2+(A-D-CF-F^2)\Gamma.
\end{align*}
Thus \(\alpha,\beta\in\mathcal R\).

For fixed \(n\geq4\), define
\begin{align*}
\mathcal S_n= H_n(Q_0)H_{n-4}(Q_0)-\Delta^2H_{n-1}(Q_0)H_{n-3}(Q_0)-\bigl((A-D-CF-F^2)\Gamma-A\Delta^2\bigr)H_{n-2}(Q_0)^2.
\end{align*}
Step~1 gives \(\mathcal S_n\in\mathcal R\), while Steps~2--4 show that its image in \(\mathcal K\) is zero. 
Since \(\mathcal R\hookrightarrow\mathcal K\) is injective, \(\mathcal S_n=0\) already in \(\mathcal R\).

Finally, apply the specialization homomorphism
\[
 \mathcal R\longrightarrow\mathcal A,\qquad
 (A,B,C,D,F)\longmapsto(a_0,b_0,c_0,d_0,f_0).
\]
By \eqref{eq:wz-coefficient-recursion}, the series \(Q_0\) specializes coefficientwise to the unique solution of
\eqref{eq:wz-cross-multiplied}.  Specializing \(\mathcal S_n=0\) proves \eqref{eq:wz-universal-somos}. 
Since the final specialization involves only elements of $\mathcal{R}$, no nonvanishing condition is required after specialization. 
This completes the proof.
\end{proof}

\begin{example}
For example, take $(a_0,b_0,c_0,d_0,f_0)=(1,1,0,0,0)$.
Then \(Q=(1+x)/(1-x^2Q)\), \(a_1=H_2(Q)=0\), and
\[\bigl(H_n(Q)\bigr)_{n\geq0} =1,1,0,-1,-1,-1,-1,0,1,\ldots .
\]
Here \((\alpha,\beta)=(1,-1)\), and
\eqref{eq:wz-universal-somos} holds, but at \(n=6\) its quotient form $H_6(Q)=\frac{H_5(Q)H_3(Q)-H_4(Q)^2}{H_2(Q)}$ has zero denominator \(H_2(Q)\) and zero numerator.  It is therefore undefined rather than a recurrence determining \(H_6(Q)\).
\end{example}

More generally, all parameterized applications of \cref{lem:sx} and \cref{thm:wz} below, including normalizations and geometric rescalings, are understood over the rational function field of the parameters. 

The following lemma can be regarded as a lifting and scaling of the \((\alpha,\beta)\) Somos-4 sequence.

\begin{lemma}\label{lem:somos-invariance}
If \((S_n)_{n\geq 0}\) satisfies the \((\alpha,\beta)\) Somos-4 identity \eqref{eq:somos4-bilinear}, then so does
\[T_n=u\,v^nS_{n+k}
\]
for every \(k\in\mathbb Z_{\geq0}\) and scalars \(u,v\).
\end{lemma}
\begin{proof}
For \(n\geq4\), direct substitution gives
\[T_nT_{n-4}-\alpha T_{n-1}T_{n-3}-\beta T_{n-2}^2
=u^2v^{2n-4}\bigl(S_{n+k}S_{n+k-4}-\alpha S_{n+k-1}S_{n+k-3}-\beta S_{n+k-2}^2\bigr)=0.
\]
This completes the proof.
\end{proof}

We prove the full two-parameter assertion in the comments section of on OEIS entry [A160702], not only its numerical specialization. 

\begin{corollary}{\em (Conjectured in \cite[A160702]{Sloane23})}\label{thm:new-a160702}
Let
\begin{equation}\label{eq:new-a160702-recurrence}
 u_0=u_1=u_2=1,\qquad u_n=ru_{n-1}+s\sum_{k=0}^{n-2}u_ku_{n-1-k}\quad(n\geq3),
\end{equation}
and let
\[ F(x)=\sum_{n\geq0}u_{n+1}x^n,\qquad h_n=H_{n+1}(F).
\]
Then the shifted Hankel transform \((h_n)_{n\geq 0}\) is an $\bigl(s^2,s^3(r+2s-2)\bigr)$ Somos-4 sequence.
\end{corollary}
\begin{proof}
By summing \eqref{eq:new-a160702-recurrence}, we obtain
\begin{equation}\label{eq:new-a160702-functional}
 F(x)=\frac{1+(1-r-s)x}{1-(r+s)x-sx^2F(x)}.
\end{equation}
Apply \cref{lem:sx} with
\[ (a,b,c,d,e,f) =(1,1-r-s,-r-s,0,-s,0).\]
It follows that \(H_{n+1}(F)=H_n(G)\), where
\[ G(x)= \frac{r+2s-1+(1-r-s)x}{1-(r+s)x+2(r+s-1)x^2+x^2(-1+(r+s-1)x)G(x)}.
\]
Thus \(G(x)\) satisfies the Wang--Zhang condition with
\begin{align*}
 a_0&=r+2s-1,\quad b_0=1-r-s,\quad c_0=-r-s,\qquad d_0=2(r+s-1),\qquad f_0=r+s-1,\\
 f_1&=\frac{r+s-1}{r+2s-1},\qquad a_1=\frac{s(rs+r+2s^2-1)}{(r+2s-1)^2},\qquad c_0+f_0+f_1=-\frac{s}{r+2s-1}.
\end{align*}
Substitution in \cref{thm:wz} gives $\alpha=s^2$ and $\beta=s^3(r+2s-2)$.
The apparent singularity at \(r+2s-1=0\), and all other degenerate specializations, are covered by \cref{rem:specialization}.
\end{proof}

\begin{corollary}{\em (Conjectured in \cite[A160703]{Sloane23})}\label{cor:new-a160703}
For \((r,s)=(1,2)\) in \cref{thm:new-a160702}, \((u_n)_{n\geq 0}\) is OEIS sequence [A160702] and $(h_n)_{n\geq0}$ is OEIS sequence [A160703].  The latter satisfies
\[h_nh_{n-4}=4h_{n-1}h_{n-3}+24h_{n-2}^2\qquad(n\geq4).
\]
\end{corollary}
\begin{proof}
The recurrence \eqref{eq:new-a160702-recurrence} with \((r,s)=(1,2)\) is the defining recurrence recorded for [A160702]. 
The first four determinants are \(1,4,20,464\), and \cref{thm:new-a160702} gives the parameters \((4,24)\).  These four initial values and the recurrence determine the resulting positive sequence inductively.  It is precisely the Hankel sequence
and recurrence recorded as [A160703].
\end{proof}

Next, we conclude this section by proving a conjecture that remained unresolved in Barry's paper \cite{Barry2022Conjectures}.

\begin{corollary}{\em (Conjectured in \cite[Page 3]{Barry2022Conjectures})}\label{thm:2022-unnumbered}
Let \(g(0)=1\) satisfy
\[g(x)=\frac1{1-\dfrac{x}{1-rx}-x^2g(x)}.
\]
Then the shifted Hankel sequence \((H_{n+1}(g))_{n\geq0}\) is a $\bigl((r-1)^2,\,4r\bigr)$ Somos-4 sequence.
\end{corollary}
\begin{proof}
Rewriting the defining equation as
\[
 g(x)=\frac{1-rx}{1-(r+1)x-x^2(1-rx)g(x)}
\]
and applying \cref{lem:sx} with
\[
 (a,b,c,d,e,f)=(1,-r,-r-1,0,-1,r)
\]
gives \(H_{n+1}(g)=H_n(G)\), where
\[
 G(x)=\frac{r+1-2rx}{1-(r+1)x+2rx^2+x^2(-1+rx)G(x)}.
\]
Thus \(G(x)\) satisfies the Wang--Zhang condition with
\[a_0=r+1,\quad b_0=-2r,\quad c_0=-r-1,\quad d_0=2r,\quad f_0=r.
\]
The auxiliary quantities are
\[
 f_1=\frac{2r}{r+1},\qquad
 a_1=\frac{r^3-r^2+3r+1}{(r+1)^2},\qquad
 c_0+f_0+f_1=\frac{r-1}{r+1}.
\]
Substitution in \cref{thm:wz} gives \(\alpha=(r-1)^2\) and \(\beta=4r\).
The singular value \(r=-1\) follows by the specialization principle in \cref{rem:specialization}.
\end{proof}

\section{The Hurwitz transform conjecture}\label{sec:new-hurwitz}

This section mainly proves a conjecture in \cite[Example~7]{Barry2012Hurwitz} concerning the Hurwitz transform and the Hankel transform. In fact, we prove a general case.
Before that, we need the prove \cref{lem:parity-decimation}.
It shows that the even and odd subsequences of an $(\alpha,\beta)$ Somos-4 sequence also forms a Somos-4 sequence.

\begin{proof}[Proof of \cref{lem:parity-decimation}]
By the Somos identity, we can verify that the following result holds:
\begin{equation}\label{eq:new-qrt-map}
 x_{n+1}x_{n-1}=\frac{\alpha x_n+\beta}{x_n^2}.
\end{equation}
Substitution of \eqref{eq:new-qrt-map} at two consecutive indices shows directly that \(J\) is independent of \(n\). 
In fact, by \eqref{eq:new-qrt-map}, we have 
$$x_nx_{n+1}=\frac{\alpha x_n+\beta}{x_{n-1}x_n}.$$
Thus we get 
$$x_nx_{n+1}+\frac{\alpha}{x_n}=\frac{\alpha}{x_{n-1}}+\frac{\alpha}{x_{n}}+\frac{\beta}{x_{n-1}x_n}.$$
Furthermore, we obtain
$$\frac{x_n^2(x_{n+1}+x_{n-1})+\alpha}{x_n}=x_n(x_{n+1}+x_{n-1})+\frac{\alpha}{x_n}=J.$$
In other words, one has
\begin{align}\label{Equation-XXN11-2}
x_{n-1}+x_{n+1}=\frac{Jx_n-\alpha}{x_n^2}.
\end{align}

Now we consider the subsequence $(S_{2m+\varepsilon})_{m\geq 0}$. According to 
$$S_{2m+2+\varepsilon}S_{2m-2+\varepsilon}=\alpha S_{2m+1+\varepsilon}S_{2m-1+\varepsilon}+\beta S_{2m+\varepsilon}^2,$$ 
we have
\begin{align}\label{Equation-Somos-4-YMS}
y_m:=\frac{S_{2m+2+\varepsilon}S_{2m-2+\varepsilon}}{S_{2m+\varepsilon}^2}=\frac{\alpha S_{2m+1+\varepsilon}S_{2m-1+\varepsilon}+\beta S_{2m+\varepsilon}^2}{S_{2m+\varepsilon}^2}=\alpha x_{2m+\varepsilon}+\beta.
\end{align}
By
\[x_{n\pm2} =\frac{\alpha x_{n\pm1}+\beta}{x_{n\pm1}^2x_n},
\]
we have
$$y_{m+1}=\alpha \frac{\alpha x_{2m+1+\varepsilon}+\beta}{x_{2m+1+\varepsilon}^2 x_{2m+\varepsilon}}+\beta=\frac{\alpha(\alpha x_{2m+1+\varepsilon}+\beta)+\beta x_{2m+1+\varepsilon}^2 x_{2m+\varepsilon}}{x_{2m+1+\varepsilon}^2 x_{2m+\varepsilon}}.$$
Similarly, we have
$$y_{m-1}=\frac{\alpha(\alpha x_{2m-1+\varepsilon}+\beta)+\beta x_{2m-1+\varepsilon}^2 x_{2m+\varepsilon}}{x_{2m-1+\varepsilon}^2 x_{2m+\varepsilon}}.$$
Observing the numerator of $y_{m+1}y_{m-1}$, we obtain the following.
\begin{small}
\begin{align*}
& (\alpha(\alpha x_{2m+1+\varepsilon}+\beta)+\beta x_{2m+1+\varepsilon}^2 x_{2m+\varepsilon}) (\alpha(\alpha x_{2m-1+\varepsilon}+\beta)+\beta x_{2m-1+\varepsilon}^2 x_{2m+\varepsilon})
\\ =& \alpha^4 x_{2m+1+\varepsilon} x_{2m-1+\varepsilon} +\alpha^3\beta(x_{2m+1+\varepsilon} +x_{2m-1+\varepsilon}) +\alpha^2\beta x_{2m+1+\varepsilon}x_{2m+\varepsilon} x_{2m-1+\varepsilon}(x_{2m+1+\varepsilon}+x_{2m-1+\varepsilon})
 \\ & \quad +\alpha^2\beta^2+\alpha\beta^2x_{2m+\varepsilon}(x_{2m+1+\varepsilon}^2+x_{2m-1+\varepsilon}^2)+\beta^2x_{2m+1+\varepsilon}^2 x_{2m+\varepsilon}^2 x_{2m-1+\varepsilon}^2.
\end{align*}
\end{small}
From Equations \eqref{eq:new-qrt-map} and \eqref{Equation-XXN11-2}, by a lengthy calculation, we obtain 
$$(\alpha(\alpha x_{2m+1+\varepsilon}+\beta)+\beta x_{2m+1+\varepsilon}^2 x_{2m+\varepsilon}) (x_{2m-1+\varepsilon}^2 x_{2m+\varepsilon})=\frac{\alpha \Lambda^2}{x_{2m+\varepsilon}}+\frac{\beta(\beta^3-\alpha^2\Lambda)}{x_{2m+\varepsilon}^2}.$$
Therefore, we have
\begin{align*}
y_{m-1}y_{m+1}&=\left(\frac{\alpha \Lambda^2}{x_{2m+\varepsilon}}+\frac{\beta(\beta^3-\alpha^2\Lambda)}{x_{2m+\varepsilon}^2}\right)\frac{1}{x_{2m+1+\varepsilon}^2 x_{2m+\varepsilon}^2 x_{2m-1+\varepsilon}^2}
\\& = \left(\frac{\alpha \Lambda^2}{x_{2m+\varepsilon}}+\frac{\beta(\beta^3-\alpha^2\Lambda)}{x_{2m+\varepsilon}^2}\right)\frac{x_{2m+\varepsilon}^2}{(\alpha x_{2m+\varepsilon} +\beta)^2}
\\ &= \left( \Lambda^2(\alpha x_{2m+\varepsilon}+\beta) -\beta \Lambda^2+\beta(\beta^3-\alpha^2\Lambda) \right)\frac{1}{(\alpha x_{2m+\varepsilon} +\beta)^2}
\\ &= \frac{\Lambda^2y_m+ \beta\bigl(\beta^3-\Lambda(\Lambda+\alpha^2)\bigr)}{y_m^2}.
\end{align*}
Now, from Equations \eqref{eq:new-qrt-map} and \eqref{Equation-Somos-4-YMS}, we see that the sequence 
\((S_{2m+\varepsilon})_{m\geq 0}\) is a
\[\left(\Lambda^2,\,\beta\bigl(\beta^3-\Lambda(\Lambda+\alpha^2)\bigr)\right)
\]
Somos-4 sequence. 
\end{proof}

Given sequences $(a_n)_{n\geq 0}$ and $(b_n)_{n\geq 0}$, define the Hurwitz matrix associated with $a_n$ and $b_{n}$ as follows: If \(n = 2m\), then the \emph{Hurwitz matrix of order \(n\)} is defined to be the matrix
\[
M_n = 
\begin{pmatrix}
a_0 & a_1 & a_2 & \cdots & a_m & \cdots & a_{2m} \\
b_0 & b_1 & b_2 & \cdots & b_m & \cdots & b_{2m} \\
0 & a_0 & a_1 & \cdots & a_{m-1} & \cdots & a_{2m-1} \\
0 & b_0 & b_1 & \cdots & b_{m-1} & \cdots & b_{2m-1} \\
\vdots & \vdots & \vdots & \ddots & \vdots & \ddots & \vdots \\
0 & 0 & 0 & \cdots & a_0 & \cdots & a_m
\end{pmatrix}.
\]
If \(n = 2m + 1\), then the \emph{Hurwitz matrix of order \(n\)} is defined to be the matrix
\[
M_n = 
\begin{pmatrix}
a_0 & a_1 & a_2 & \cdots & a_m & \cdots & a_{2m+1} \\
b_0 & b_1 & b_2 & \cdots & b_m & \cdots & b_{2m+1} \\
0 & a_0 & a_1 & \cdots & a_{m-1} & \cdots & a_{2m} \\
0 & b_0 & b_1 & \cdots & b_{m-1} & \cdots & b_{2m} \\
\vdots & \vdots & \vdots & \ddots & \vdots & \ddots & \vdots \\
0 & 0 & 0 & \cdots & a_0 & \cdots & a_{m+1} \\
0 & 0 & 0 & \cdots & b_0 & \cdots & b_{m+1}
\end{pmatrix}.
\]

The sequence of determinants $\mathbb{H}_n = \det (M_n)$ ($n\geq 0$) is called the \emph{Hurwitz transform} of the sequences \(a_n\) and \(b_n\).
We shall sometimes write \(\mathbb{H}_n(a_n, b_n)\) or \(\mathbb{H}_n(a, b)\) for the transform of \(a_n\) and \(b_n\), to make the dependence on \(a_n\) and \(b_n\) more explicit. It is clear that if \(a_n\) and \(b_n\) are integer sequences, then \(\mathbb{H}_n\) is an integer sequence. 
Recall that $0^n=1$ if $n=0$, and $0^n=0$ otherwise.
We define the \emph{Hurwitz transform of a single sequence} $a_n$ to be the Hurwitz transform of the pair $(a_n,0^n)$.

We now consider the Hurwitz transform of the following sequence.
For a parameter \(s\), define
\[u_0=u_1=u_2=1,\qquad u_n=u_{n-1} +s\sum_{k=0}^{n-2}u_ku_{n-1-k}\quad(n\geq3),
\]
and let
\[ F_s(x)=\sum_{n\geq0}u_{n+1}x^n,\qquad F_s^*(x)=\frac{F_s(x)-1}{x}.
\]
Note that when $s=2$, the sequence $(u_n)_{n\geq 0}$ is [A160702] on OEIS \cite{Sloane23}.
Let
\[\mathcal U_n=\mathbb{H}_n\bigl((u_{n+1})_{n\geq0}, (0^n)_{n\geq0}\bigr).
\]
By Barry's determinant factorization \cite[Example~6 and Proposition~1]{Barry2012Hurwitz}, it is known that $\mathcal U_n$ can be represented by the following Hankel determinants.
\begin{equation}\label{eq:new-hurwitz-interlacing}
 \mathcal U_{2m}=H_{m+1}(F_s),\qquad \mathcal U_{2m+1}=(-1)^{m+1}H_{m+1}(F_s^*).
\end{equation}

\begin{theorem}\label{thm:new-hurwitz}
The Hurwitz transform of a single sequence $u_{n+1}$ (i.e., the sequence $(\mathcal U_n)_{n\geq 0}$) is a \((-s,s)\) Somos-4 sequence.
\end{theorem}
\begin{proof}
We work over the rational-function field \(\mathbb Q(s)\). 
The even determinants in \eqref{eq:new-hurwitz-interlacing} are covered by \cref{thm:new-a160702} with \(r=1\).
The sequence $D_m= \mathcal U_{2m}=H_{m+1}(F_s)$ is an \(\bigl(s^2,s^3(2s-1)\bigr)\) Somos-4 sequence.

By summing the recurrence for \(u_n\), we obtain 
\[ F_s=\frac{1-sx}{1-(s+1)x-sx^2F_s},
\]
and therefore
\begin{equation}\label{eq:new-hurwitz-shifted-f}
 F_s^*=\frac{1+sx}{1-(s+1)x-2sx^2-sx^3F_s^*}.
\end{equation}
By applying \cref{lem:sx} to \eqref{eq:new-hurwitz-shifted-f} with
\[
 (a,b,c,d,e,f)=(1,s,-s-1,-2s,0,-s),
\]
we obtain \(H_{m+1}(F_s^*)=H_m(Q_s)\), where
\[Q_s=\frac{-s(2s-1)+4s^3x}{1-(s+1)x-4s^2x^2+x^2(-1-sx)Q_s}.
\]
This satisfies the Wang--Zhang condition in \cref{thm:wz} with
\begin{align*}
 &a_0=-s(2s-1),\quad b_0=4s^3,\quad c_0=-s-1,\quad
 d_0=-4s^2,\quad f_0=-s,\\
 & f_1=\frac{4s^2}{2s-1},\qquad
 a_1=-\frac{s(6s-1)}{(2s-1)^2},\qquad
 c_0+f_0+f_1=\frac1{2s-1}.
\end{align*}
Therefore, by \cref{thm:wz}, the sequence $D_m^*= H_{m+1}(F_s^*)$ is an \(\bigl(s^2,s^3(2s-1)\bigr)\) Somos-4 sequence.
Thus, the sequence $O_m=(-1)^{m+1}D_m^*$ is an \(\bigl(s^2,s^3(2s-1)\bigr)\) Somos-4 sequence by \cref{lem:somos-invariance}.

The first four terms of $D_m$ and $O_m$ are
\begin{align*}
 (D_0,D_1,D_2,D_3) &=\bigl(1,2s,s^2(2s+1),s^4(8s^2-2s+1)\bigr),\\
 (O_0,O_1,O_2,O_3) &=\bigl(-1,-s(2s-1),s^3(6s-1),4s^7(2s-3)\bigr).
\end{align*}
Now let \((S_n)_{n\geq 0}\) be the \((-s,s)\) Somos-4 sequence with
\[(S_0,S_1,S_2,S_3)=\bigl(1,-1,2s,-s(2s-1)\bigr).
\]
For this sequence, we have the parameters in \cref{lem:parity-decimation} as follows:
\[
 x_1=2s,\qquad x_2=\frac{2s-1}{4s},\qquad
 J=-s-1,\qquad \Lambda=-s.
\]
By \cref{lem:parity-decimation}, we see that the sequence \((S_{2m+\varepsilon})_{m\geq 0}\) ($\varepsilon \in\{0,1\}$) is a
$\bigl(s^2,s^3(2s-1)\bigr)$ Somos-4 sequence. By the recurrence, we get
\begin{align*} 
S_4=s^2(2s+1),\qquad S_5=s^3(6s-1),\qquad S_6=s^4(8s^2-2s+1),\qquad S_7=4s^7(2s-3).
\end{align*}
Hence the even and odd subsequences of \(S_n\) have the same four initial values as \(D_m\) and \(O_m\), respectively. 
Thus, we have  $S_{2m}=D_m$ and $S_{2m+1}=O_m$.
By Equation \eqref{eq:new-hurwitz-interlacing}, we obtain \(\mathcal U_n=S_n\).
This completes the proof.
\end{proof}

\begin{corollary}{\em (Conjectured in \cite[Example~7]{Barry2012Hurwitz})}\label{Corollary-Hurwitz}
For \(s=2\), so that the sequence \((u_n)\) is [A160702] on OEIS, the Hurwitz transform in \eqref{eq:new-hurwitz-interlacing} is a \((-2,2)\) Somos-4 sequence.
\end{corollary}

\section{Catalan-type recurrences}\label{sec:catalan-2021}

This section resolves the determinant conjectures in \cite{Barry2021Catalan}. 
We retain the parameters used in that paper, but continue to use the convention \(H_0=1\).

We first record an evaluation that is useful for the first family.
Before that, we need some results on orthogonal polynomials.

\subsection{Orthogonal polynomials}

We recall \cref{lem:jacobi-product}, which is closely related to orthogonal polynomials.

A polynomial sequence $(p_n(x))_{n\geq 0}$ is called \emph{orthogonal} if $p_n(x)$ has degree $n$ and there exists a linear functional $L$ such that $L(p_n(x)p_m(x))=\delta_{m,n}c_n$ for some sequence $(c_n)_{n\geq 0}$ of nonzero numbers, where $\delta_{m,n}$ denotes the Kronecker delta (i.e., $\delta_{m,n}=1$ if $m=n$ and $\delta_{m,n}=0$ otherwise).

\begin{lemma}{\em (\cite{Viennot}, \cite[Theorem 50.1]{Wall2000}, \cite[Theorem 12]{Krattenthaler1999}, \cite[Theorem 29]{Krattenthaler05})}
Let $(p_n(x))_{n \geq 0}$ be a sequence of monic polynomials, the polynomial $p_n(x)$ having degree $n$. Then the sequence 
$(p_n(x))_{n \geq 0}$ is orthogonal if and only if there exist sequences $(a_n)_{n \geq 1}$ and $(b_n)_{n \geq 1}$, with $b_n \neq 0$ for all $n \geq 1$, such that the three-term recurrence  
\begin{align*}
p_{n+1}(x) = (a_n + x)p_n(x) - b_n p_{n-1}(x), \quad \text{for } n \geq 1,
\end{align*}
holds, with initial conditions $p_0(x) = 1$ and $p_1(x) = x + a_0$.
\end{lemma}

\begin{lemma}{\em (\cite{Viennot}, \cite[Theorem 51.1]{Wall2000}, \cite[Theorem 13]{Krattenthaler1999}, \cite[Theorem 29]{Krattenthaler05})} \label{Lemma-Orthogonal-L}
Let $(p_n(x))_{n \geq 0}$ be a sequence of monic polynomials, the polynomial $p_n(x)$ having degree $n$, which is orthogonal with respect to some functional $L$. Let  
\begin{align}\label{Equation-UV-Orthogonal}
p_{n+1}(x) = (-u_n + x)p_n(x) - v_np_{n-1}(x)
\end{align}
be the corresponding three-term recurrence. Then the generating function $\sum_{k\geq 0} \mu_k x^k$ for the moments $\mu_k = L(x^k)$ satisfies \eqref{Equation-FX-230-UV} with the $-u_i$'s and $v_i$'s being the coefficients in the three-term recurrence \eqref{Equation-UV-Orthogonal}.
\end{lemma}

The following theorem is a main result of this section.

\begin{theorem}\label{lem:constant-tail-border}
If the power series $F(x)$ and $G(x)$ satisfy
\begin{align}\label{Equation-GX-Contine}
G(x)=\frac{1}{1-ux-vx^2G(x)},\qquad F(x)=1+\lambda xG(x),
\end{align}
then, for \(n\geq0\), we have
\begin{align*}
H_{n+1}(F)=\lambda^n v^{\binom{n}{2}}[x^n]\frac{1-\lambda x}{1-ux+vx^2}.
\end{align*}
\end{theorem}
\begin{proof}
First assume \(\lambda v\ne0\).  The general case then follows by clearing the harmless intermediate denominators and specializing the resulting polynomial identity. Let \(G(x)=\sum_{j\geq 0} g_jx^j\), and let \(A_n=[x^n](1-ux+vx^2)^{-1}\), with \(A_{-1}=0\).
By \cref{lem:jacobi-step}, we have \(H_n(G)=v^{\binom n2}\). 
We now consider the Hankel determinant of $F(x)$ as follows:
\[H_{n+1}(F)=\det\!\begin{pmatrix}
 1 & \lambda g_0 &\cdots & \lambda g_{n-1}\\
 \lambda g_0 & \lambda g_1 & \cdots &\lambda g_n\\
 \vdots & \vdots & & \vdots \\
 \lambda g_{n-1} & \lambda g_n & \cdots & \lambda g_{2n-1}
 \end{pmatrix}=\lambda^{n+1}
\det\!\begin{pmatrix}
 \lambda^{-1}&g_0&\cdots&g_{n-1}\\
 g_0&g_1&\cdots&g_n\\
 \vdots&\vdots&&\vdots\\
 g_{n-1}&g_n&\cdots&g_{2n-1}
 \end{pmatrix}.
\]
For convenience, let $M$ denote the bordered matrix defined as follows:
\begin{align*}
    M = \begin{pmatrix} 
    \lambda^{-1} & g_0 & g_1 & \dots & g_{n-1} \\ 
    g_0 & g_1 & g_2 & \dots & g_n \\ 
    g_1 & g_2 & g_3 & \dots & g_{n+1} \\ 
    \vdots & \vdots & \vdots & \ddots & \vdots \\ 
    g_{n-1} & g_n & g_{n+1} & \dots & g_{2n-1} 
    \end{pmatrix}.
\end{align*}

Define a linear functional $L$ on the space of polynomials $\mathbb{R}[x]$ by its action on the standard basis monomials: $L(x^k) = g_k$. The generating function $G(x)$ takes the form of a Jacobi continued fraction (the left-hand side of \eqref{Equation-GX-Contine}). By \cref{Lemma-Orthogonal-L}, the sequence of moments $(g_k)_{k\geq 0}$ corresponds to a sequence of monic orthogonal polynomials $(P_k(x))_{k \ge 0}$ satisfying the three-term recurrence relation
\begin{align*}
    P_{k+1}(x) = (x-u)P_k(x) - v P_{k-1}(x), \quad \text{for } k \ge 1,
\end{align*}
with initial conditions $P_0(x) = 1$ and $P_1(x) = x - u$. The orthogonality condition gives $L(P_m(x) P_k(x)) = 0$ for $m \neq k$. 

We now prove that $L(P_k^2(x)) = v^k$.
By multiplying the recurrence by $P_{k-1}(x)$ and applying $L$, we obtain
\begin{align*}
   L(P_{k-1}(x)P_{k+1}(x)) =L(P_{k-1}(x) (x-u)P_k(x)) -L(P_{k-1}(x) v P_{k-1}(x))
\end{align*}
By the orthogonality condition, we have $0=L(xP_{k-1}(x)P_k(x))-vL(P_{k-1}(x)^2)$.
Since $P_{k-1}(x)$ is a monic polynomial of degree $k-1$, multiplying it by $x$ yields a monic polynomial of degree $k$.
Thus, $xP_{k-1}(x)$ has the form 
$$xP_{k-1}(x)=P_k(x)+\sum_{i=0}^{k-1}c_i P_i(x)$$
for some constants $c_i$.
Again by orthogonality, we obtain $L(P_k(x)^2)=vL(P_{k-1}^2)$.
It follows by iteration that $L(P_k(x)^2)=v^k$.

Now we compute the determinant $\det(M)$. We perform a change of basis from the standard monomials $\{1, x, \dots, x^{n-1}\}$ to the orthogonal polynomials $\{P_0(x), P_1(x), \dots, P_{n-1}(x)\}$. Let $C$ be the $n \times n$ transition matrix such that $P_{j-1}(x) = \sum_{i=1}^n C_{i,j} x^{i-1}$. Since $P_k(x)$ is monic, $C$ is unit upper-triangular, implying $\det(C) = 1$. We define the $(n+1) \times (n+1)$ block diagonal transformation matrix 
$$\tilde{C} = \left(\begin{matrix} 1 & \mathbf{0} \\ \mathbf{0} & C \end{matrix}\right).$$
Therefore, we have $\det(M) = \det(M')$, where $M' = \tilde{C}^T M \tilde{C}$.
The elements of the transformed matrix $M'$ are computed using the linearity of $L$. 

Case 1: For the top-left entry, it is clear that $M'_{0,0} = \lambda^{-1}$. 

Case 2: For the first row ($j \ge 1$), the entries become $M'_{0,j} = L(1 \cdot P_{j-1}(x))$. By orthogonality with $P_0(x)=1$, we have $M'_{0,1} = g_0 = 1$ and $M'_{0,j} = 0$ for $j \ge 2$. By symmetry, the first column is identical.

Case 3: For the lower-right $n \times n$ block ($i, j \ge 1$), we consider the elements $M'_{i,j}$.
We know that $P_{j-1}(x)=\sum_{s=1}^n C_{s,j}x^{s-1}$ and $P_{i-1}(x)=\sum_{r=1}^n C_{r,i}x^{r-1}$.
By $M_{r,s}=L(x\cdot x^{r-1}\cdot x^{s-1})$, we have
$$M'_{i,j}=\sum_{r=1}^n\sum_{s=1}^n (C^T)_{i,r} M_{r,s} C_{s,j}=\sum_{r=1}^n\sum_{s=1}^n C_{r,i}L(x\cdot x^{r-1}\cdot x^{s-1}) C_{s,j}.$$
According to the linearity of $L$, we obtain
$$M'_{i,j}=L\left( x\cdot \left( \sum_{s=1}^n C_{s,j}x^{s-1}\right)\cdot \left( \sum_{r=1}^n C_{r,i}x^{r-1}\right)\right)
=L(xP_{j-1}(x)P_{i-1}(x)).$$
By the orthogonality relations, $M'_{i,j}$ is non-zero only when $i$ and $j$ are adjacent or equal. 
Using the recurrence $x P_k(x) = P_{k+1}(x) + u P_k(x) + v P_{k-1}(x)$, we derive the following: 
$$M'_{i,i} =uL(P_{i-1}(x)^2)= u v^{i-1}, \quad M'_{i,i+1} =L(P_i(x)^2)= v^i,\quad  \text{and}\quad M'_{i+1,i} =L(P_i(x)^2)= v^i.$$ 

Thus, $M'$ takes the form
\begin{align*}
    M' = \begin{pmatrix} 
    \lambda^{-1} & 1 & 0 & 0 & \dots \\ 
    1 & u & v & 0 & \dots \\ 
    0 & v & uv & v^2 & \dots \\ 
    0 & 0 & v^2 & uv^2 & \dots \\ 
    \vdots & \vdots & \vdots & \vdots & \ddots 
    \end{pmatrix}.
\end{align*}
Expanding $\det(M')$ along its first row yields
\begin{align*}
    \det(M') = \lambda^{-1} \det(T) - \det(T^{(1)}),
\end{align*}
where $T$ is the lower-right $n \times n$ tridiagonal block, and $T^{(1)}$ is the $(n-1) \times (n-1)$ submatrix obtained by removing the first row and column of $T$. 

We can factor $T = D J_n$, where $D = \mathrm{diag}(1, v, \dots, v^{n-1})$ and $J_n$ is the classical $n \times n$ Jacobi matrix associated with the recurrence. That is 
$$J_n = \begin{pmatrix} u & v & 0 & \dots \\ 1 & u & v & \dots \\ 0 & 1 & u & \dots \\ \vdots & \vdots & \vdots & \ddots \end{pmatrix}.$$
Clearly, $\det(D) = v^{\binom{n}{2}} = H_n(G)$. Expanding $\det(J_n)$ along its first row shows it satisfies the recurrence 
$$\det(J_n)=u\det(J_{n-1})-v\det(J_{n-2}),$$
with initial condition $\det(J_1)=u$ and $\det(J_2)=u^2-v$.
By the coefficients $A_n = [x^n](1 - ux + vx^2)^{-1}$, we know that
$$A_n = u A_{n-1} - v A_{n-2},$$
with initial condition $A_1=u$ and $A_2=u^2-v$.
Therefore, we obtain $\det(J_n)=A_n$ and $\det(T) = H_n(G) A_n$.

Similarly, we have $T^{(1)} = D^{(1)} J_{n-1}$, where $D^{(1)} = \mathrm{diag}(v, v^2, \dots, v^{n-1})$. 
The determinant evaluates to $\det(T^{(1)}) = H_n(G) A_{n-1}$. 

Therefore, we obtain
\begin{align*}
\det(M)=\det(M') = H_n(G) \left( \lambda^{-1} A_n - A_{n-1} \right).
\end{align*}
Finally, we get 
\begin{align*}
    H_{n+1}(F) &= \lambda^{n+1} v^{\binom{n}{2}} \left( \lambda^{-1} A_n - A_{n-1} \right) \\
    &= \lambda^n v^{\binom{n}{2}} \left( A_n - \lambda A_{n-1} \right) \\
    &= \lambda^n v^{\binom{n}{2}} [x^n] \frac{1 - \lambda x}{1 - ux + vx^2}.
\end{align*}
This completes the proof.
\end{proof}

\begin{theorem}\label{Theorem-SH-FFN-UV}
The shifted Hankel transform $(H_{n+1}(F))_{n\geq 0}$ of generating function $F(x)$ in \cref{lem:constant-tail-border} forms a $\bigl(u^2v^2,\,v^3(v-u^2)\bigr)$ Somos-4 sequence.
\end{theorem}
\begin{proof}
Define \(Q(x)\) by
\begin{align*}
 \frac{1}{F(x)}=1-\lambda x-x^2Q(x).
\end{align*}
By \cref{lem:jacobi-step}, we have \(H_{n+1}(F)=H_n(Q)\).
By \eqref{Equation-GX-Contine}, we obtain
\begin{align}\label{eq:2021-second-Q}
Q(x)=\frac{(u\lambda-\lambda^2)+(\lambda^3-u\lambda^2+v\lambda)x}{1-ux+(2u\lambda-2\lambda^2-2v)x^2+x^2(-1+(u-\lambda
-\frac{v}{\lambda})x)Q(x)}.
\end{align}
Equation \eqref{eq:2021-second-Q} satisfies the Wang--Zhang condition with parameters
\[a_0=u\lambda-\lambda^2,\quad b_0=\lambda^3-u\lambda^2+v\lambda,\quad c_0=-u,\quad d_0=2(u\lambda-\lambda^2-v),\quad f_0=u-\lambda-\frac{v}{\lambda}.
\]
They give
\[f_1=\frac{\lambda^2-u\lambda+v}{\lambda-u},\qquad a_1=\frac{v(u^2-u\lambda-v)}{(u-\lambda)^2},\qquad c_0+f_0+f_1=\frac{vu}{\lambda(\lambda-u)}.
\]
By \cref{thm:wz}, the sequence \((H_{n+1}(F))_{n\geq0}\) is a $\bigl(u^2v^2,\,v^3(v-u^2)\bigr)$ Somos-4 sequence.
We compute the initial values directly from the definition $F(x)=1+\lambda xG(x)$ and the expansion of $G(x)$:
\begin{align*}
H_0(F)&=H_1(F)=1,\quad H_2(F)=\lambda(u-\lambda),\quad  H_3(F)=-\lambda^2 v(\lambda u-u^2+v), 
\\ H_4(F)&=-\lambda^3v^3(\lambda u^2-u^3-\lambda v+2uv).
\end{align*}
This completes the proof.
\end{proof}

\subsection{Second and third order recurrences}

We define the parameter sequence $a_n(p,s,t)$ by the second-order recurrence 
$$a_n=sa_{n-1}+t\sum_{k=0}^{n-3}a_{k+1}a_{n-k-2},\quad (n\geq 2)$$
with the initial conditions $a_0=1$ and $a_1=p$.
When $(p,s,t)=(1,1,1)$, the sequence $a_{n+1}(1,1,1)$ are the Motzkin numbers.
When $(p,s,t)=(1,2,1)$, the sequence $a_n(1,2,1)$ are the Catalan numbers.
When $(p,s,t)=(1,1,2)$, the sequence $a_n(1,1,2)$ is the sequence [A025235] on OEIS. This sequence counts Motzkin paths with the up-steps in two colors.

Let \(U(x)=\sum_{n\geq0}a_n(p,s,t)x^n\) be the generating function of $a_n(p,s,t)$.
Its defining recurrence is equivalent to
\begin{equation}\label{eq:2021-second-U}
 U(x)=1+px+sx(U(x)-1)+tx(U(x)-1)^2.
\end{equation}

\begin{theorem}{\em (Conjectured in \cite[Conjecture 14]{Barry2021Catalan})}\label{thm:2021-c14}
The shifted Hankel determinant \(H_{n+1}(U)\) is given by 
\begin{align*}
t^{\binom n2}p^{\binom{n+1}2}[x^n]\frac{1-px}{1-sx+ptx^2}.
\end{align*}
\end{theorem}
\begin{proof}
Writing $G(x)=\frac{U(x)-1}{px}$,
we obtain 
$$G(x)=\frac{1}{1-sx-ptx^2G(x)}.$$
Apply \cref{lem:constant-tail-border} with \((u,v,\lambda)=(s,pt,p)\), and note that \(p^n(pt)^{\binom n2}=t^{\binom n2}p^{\binom{n+1}2}\). This completes the proof.
\end{proof}

\begin{theorem}{\em (Conjectured in \cite[Conjecture 15]{Barry2021Catalan})} \label{thm:2021-c15}
The Hankel transform of generating function $U(x)$ forms a $\bigl((pst)^2,\,(pt)^3(pt-s^2)\bigr)$ Somos-4 sequence.
\end{theorem}
\begin{proof}
This result follows from \cref{Theorem-SH-FFN-UV} with $(u,v,\lambda)=(s,pt,p)$.
\end{proof}

When $(p,s,t)=(1,-1,-1)$, the Hankel transform of $a_n(1,-1,-1)$ is the signed Fibonacci sequence:
$$1,1,-2,-3,5,8,-13,\ldots$$
with general term $(-1)^{\binom{n}{2}}F_{n+1}$.
Here $F_n$ is the Fibonacci sequence, see [A000045] on OEIS~\cite{Sloane23}.

We define the parameter sequence $a_n(p,q,r,s,t)$ by the third-order recurrence 
$$a_n=ra_{n-1}+sa_{n-2}+t\sum_{k=1}^{n-3}a_{k}a_{n-k-2},\quad (n\geq 3)$$
with the initial conditions $a_0=1$, $a_1=p$, and $a_2=q$.
The sequence $a_n(1,2,2,1,1)$ is [A086581] on OEIS, which counts the number of Dyck paths of semi-length $n$ that avoid DDUU.
The sequence $a_n(1,2,1,2,1)$ coincides with the Motzkin numbers.

Let \(U(x)=\sum_{n\geq0}a_n(p,q,r,s,t)x^n\) denote the generating function of $a_n(p,q,r,s,t)$.

\begin{theorem}{\em (Conjectured in \cite[Conjecture 24]{Barry2021Catalan})}\label{thm:2021-c24}
Let \(G(x)=(U(x)-1)/(px)\) and $p\neq 0$. The shifted Hankel transform \((H_{n+1}(G))_{n\geq 0}\) of \((a_{n+1}/p)_{n\geq0}\) forms a
\[\bigl((pt)^2,\,-t^2(p^2s+pqr-q^2)\bigr)
\]
Somos-4 sequence.
\end{theorem}
\begin{proof}
By the recurrence and the initial conditions of $a_n(p,q,r,s,t)$, we have
\begin{equation}\label{eq:2021-third-U}
 U(x)=1+px+qx^2+rx(U(x)-1-px)+sx^2(U(x)-1)+tx^2(U(x)-1)^2.
\end{equation}
Assume \(p\ne0\), as is required by the normalized sequence \((a_{n+1}/p)\). 
For \(G(x)=(U(x)-1)/(px)\), this becomes
\begin{equation}\label{eq:2021-third-G}
 G(x)=\frac{1+(q/p-r)x}{1-rx-sx^2-ptx^3G(x)}.
\end{equation}
Apply \cref{lem:sx} to
\eqref{eq:2021-third-G}.  The transformed series \(Q(x)\), for which \(H_{n+1}(G)=H_n(Q)\), is given by
\begin{align*}
Q(x)=\frac{\frac{sp^2+pqr-q^2}{p^2}+\frac{p^4t-p^2qs-pq^2r+q^3}{p^3}x}{1-rx+\frac{sp^2+2pqr-2q^2}{p^2}x^2
+x^2(-1-(\frac{q}{p}-r)x)Q(x)}.
\end{align*}
Set \(\Psi=p^2s+pqr-q^2\). Thus $Q(x)$ satisfies the Wang--Zhang condition with
\[ a_0=\frac{\Psi}{p^2},\quad b_0=pt-\frac{q\Psi}{p^3},\quad c_0=-r,\quad
 d_0=\frac{2\Psi}{p^2}-s,\quad f_0=r-\frac qp.
\]
Writing \(f_1=-b_0/a_0\), direct simplification gives
\[ c_0+f_0+f_1=-\frac{p^3t}{\Psi},\qquad
 a_1=(f_0-f_1)(c_0+f_0+f_1).
\]
Consequently the second term in \eqref{eq:wz-beta} vanishes, and \cref{thm:wz} gives \(\alpha=p^2t^2\) and \(\beta=-t^2\Psi\).
This completes the proof.
\end{proof}

\begin{theorem}{\em (Conjectured in \cite[Conjecture 25]{Barry2021Catalan})}\label{thm:2021-c25}
For \((r,s,t)=(1,q-p+1,1)\), the Hankel transform of \((a_n(p,q,r,s,t))\), namely \((H_n(U))_{n\geq0}\), is a
\[ \bigl(p^2,\,p^3-pq+q^2-p^2(1+q)\bigr)
\]
Somos-4 sequence.
\end{theorem}
\begin{proof}
Set \((r,s,t)=(1,q-p+1,1)\) in \eqref{eq:2021-third-U}, and define \(Q(x)\) by \(1/U(x)=1-px-x^2Q(x)\).
Then we obtain
\begin{equation}\label{eq:2021-c25-Q}
 Q(x)=\frac{q-p^2+p(p^2-q+1)x}{1-x+(-2p^2+p+q-1)x^2+x^2(-1+(1-p)x)Q(x)}.
\end{equation}
By \cref{lem:jacobi-step}, we have \(H_{n+1}(U)=H_n(Q)\).  In \eqref{eq:2021-c25-Q}, the $Q(x)$ satisfies the Wang--Zhang condition with
\[ a_0=q-p^2,\quad b_0=p(p^2-q+1),\quad c_0=-1,\quad d_0=-2p^2+p+q-1,\quad f_0=1-p.
\]
Insertion in \cref{thm:wz} gives exactly the displayed parameters for \((H_{n+1}(U))_{n\geq0}\), and hence proves the required identities for \((H_n(U))_{n\geq0}\) at \(n\geq5\).  Direct expansion of \eqref{eq:2021-third-U} verifies the remaining identity at \(n=4\). This completes the proof.
\end{proof}

\section{Elliptic parameter families}\label{Section-Elliptic-PF}

The generating functions mentioned in the following two conjectures are closely related to the family of elliptic curves 
$$E_t: y^2+4xy+y=x^3+(t-1)x^2+tx;$$
see \cite[Section 8]{Barry2021Catalan} for related material.

\begin{theorem}{\em (Conjectured in \cite[Conjecture 43]{Barry2021Catalan})} \label{thm:2021-c43}
Let $G(x)$ be the generating function 
$$G(x)=\frac{1+3t+t^2-x}{1+2(t+2)x}\mathcal{C}\left(\frac{x^2(1+3t+t^2-x)}{(1+2(t+2)x)^2}\right).$$
Then the Hankel transform of \(G(x)\) is an \((\alpha,\beta)\) Somos-4 sequence, where
\begin{align*}
\alpha&=(2t^3+10t^2+14t+5)^2,\\
\beta&=-3t^8-40t^7-222t^6-666t^5-1173t^4-1230t^3-740t^2-232t-29.
\end{align*}
\end{theorem}
\begin{proof}
Let \(m=t^2+3t+1\).
By the Catalan generating function \(\mathcal{C}(x)=1+x\mathcal{C}(x)^2\), we obtain
\begin{align*}
 G(x)=\frac{m-x}{1+2(t+2)x-x^2G(x)}.
\end{align*}
Thus \(G(x)\) satisfies the Wang--Zhang condition with
\[ a_0=m,\quad b_0=-1,\quad c_0=2(t+2),\quad d_0=f_0=0.
\]
The formulas of \cref{thm:wz} give $\alpha=(2(t+2)m+1)^2$ and
\[
 \beta=-m(3t^6+31t^5+126t^4+257t^3+276t^2+145t+29),
\]
which expand to the claimed expressions. This completes the proof.
\end{proof}

\begin{theorem}{\em (Conjectured in \cite[Conjecture 44]{Barry2021Catalan})}\label{thm:2021-c44}
Let $F(x)$ be the generating function 
$$F(x)=\frac{1+(2t+5)x}{1+2(t+3)x+(t+2)(t+3)x^2}\mathcal{C}\left(\frac{x^3(1+(2t+5)x)}{(1+2(t+3)x+(t+2)(t+3)x^2)^2}\right).$$
Then the Hankel sequence \((H_n(F))_{n\geq0}\) is a \((1,t^2+3t+1)\) Somos-4 sequence.
\end{theorem}
\begin{proof}
Let \(m=t^2+3t+1\).
By the Catalan generating function \(\mathcal{C}(x)=1+x\mathcal{C}(x)^2\), we obtain
\[F(x)=\frac{1+(2t+5)x}{1+2(t+3)x+(t+2)(t+3)x^2-x^3F(x)}.
\]
One application of \cref{lem:sx} gives \(H_{n+1}(F)=H_n(Q)\), where 
\begin{align*}
Q(x)=\frac{(-t^2-3t-1)+(-t^2-3t)x}{1+(2t+6)x+(-t^2-t+4)x^2+x^2(-1-(2t+5)x)Q(x)}.
\end{align*}
Here \(Q(x)\) satisfies the Wang--Zhang condition with
\[a_0=-m,\quad b_0=-t(t+3),\quad c_0=2(t+3),\quad d_0=-t^2-t+4,\quad f_0=-2t-5.
\]
Here
\[ f_1=-\frac{t(t+3)}m,\qquad a_1=-\frac{2t^3+10t^2+14t+5}{m^2},\qquad c_0+f_0+f_1=\frac1m.
\]
Equations \eqref{eq:wz-alpha}--\eqref{eq:wz-beta} simplify to \((\alpha,\beta)=(1,m)\).  This proves the identities for
\((H_n(F))_{n\geq0}\) at every \(n\geq5\); direct expansion of the defining equation verifies the remaining identity at \(n=4\). 
This completes the proof.
\end{proof}

The generating function mentioned in the following conjecture is closely related to the family of elliptic curves \cite[Section 9]{Barry2021Catalan}
$$E: y^2+axy+y=x^3+bx^2+x.$$

\begin{theorem}{\em (Conjectured in \cite[Page 40]{Barry2021Catalan})}\label{thm:2021-p39}
Let $F(x)$ be the generating function 
$$F(x)=\frac{1-(a+1)x}{1-ax-bx^2}\mathcal{C}\left(\frac{-x^3(1-(a+1)x)}{(1-ax-bx^2)^2}\right).$$
Then the Hankel sequence \((H_n(F))_{n\geq0}\) is a \((1,a-b+1)\) Somos-4 sequence.
\end{theorem}
\begin{proof}
By the Catalan generating function \(\mathcal{C}(x)=1+x\mathcal{C}(x)^2\), we obtain
\begin{equation}\label{eq:2021-p39-F}
 F(x)=\frac{1-(a+1)x}{1-ax-bx^2+x^3F(x)}.
\end{equation}
Applying \cref{lem:sx} to \eqref{eq:2021-p39-F}, we obtain \(H_{n+1}(F)=H_n(Q)\), where
\begin{align*}
Q(x)=\frac{(-a+b-1)+(-a+b-2)x}{1-ax+(-2a+b-2)x^2+x^2(-1-(-a-1)x)Q(x)}.
\end{align*} 
Thus \(Q(x)\) satisfies the Wang--Zhang condition with 
\[ a_0=-a+b-1,\quad b_0=-a+b-2,\quad c_0=-a, \quad d_0=-2a+b-2,\quad f_0=a+1.
\]
Writing \(M=a-b+1\), the auxiliary quantities are
\[ f_1=-\frac{M+1}{M},\qquad a_1=-\frac{a^2-ab+3a-2b+3}{M^2},\qquad c_0+f_0+f_1=-\frac1M.
\]
By Equations \eqref{eq:wz-alpha}--\eqref{eq:wz-beta}, we obtain \(\alpha=1\) and \(\beta=M\).  This proves the identities for \((H_n(F))_{n\geq0}\) at every \(n\geq5\); direct expansion of \eqref{eq:2021-p39-F} verifies the remaining identity at \(n=4\).
This completes the proof.
\end{proof}

The generating function mentioned in the following conjecture is closely related to the family of elliptic curves \cite[Section 3]{Barry2023Elliptic}
$$E: y^2+axy+by=x^3+cx^2+dx.$$

Suppose $b\neq 0$. Let $g(x)$ denote the generating function given as follows.
\[g(x)=\frac{2b^4x}{\sqrt{P(x)}+R(x)},
\]
where
\begin{align*}
 R(x)={}&-1+\bigl(2(b^4-d-1)-ab\bigr)x-\bigl(ab(d+1)+2b^4-b^2c+(d+1)^2\bigr)x^2,
\\ P(x)={}&R(x)^2+4b^4xK(x),
\\ K(x)={}&1+(ab-b^4+2d+1)x+(abd+2b^4-b^2c+d^2-2)x^2\\
&+(-abd-2ab-b^4+b^2c-d^2-4d-2)x^3.
\end{align*}
Note $\sqrt{P(0)}=1$.
Consequently, by rationalizing, we have
\begin{equation}\label{eq:elliptic-first-quadratic}
 K(x)g(x)^2+R(x)g(x)-b^4x=0.
\end{equation}

Let $\delta=abd-b^2c+d^2$ and $h_n=H_{n+1}(g)$.

\begin{theorem}{\em (Conjectured in \cite[Conjecture~2]{Barry2023Elliptic})}\label{thm:elliptic-2023}
Assume that \(b\ne0\).
The normalized determinants $S_n=b^{-n^2+2n}h_n$ satisfy the \((b^2,abd-b^2c+d^2)\) Somos-4 identity. 
\end{theorem}
\begin{proof}
We consider \cref{lem:jacobi-step} by setting 
\begin{align*}
 g(x)=\frac1{1-x-x^2q(x)}.
\end{align*}
By substituting in \eqref{eq:elliptic-first-quadratic}, we obtain
\begin{equation}\label{eq:elliptic-q}
q(x)=\frac{1+(ab+2d+1)x}{1+(ab+2d+2)x+Lx^2-b^4x^3q(x)},
\end{equation}
where
\[L=abd+ab-b^2c+d^2+2d+1.
\]
By \cref{lem:jacobi-step}, we have \(H_n(g)=H_{n-1}(q)\) for \(n\geq1\), and hence
\(h_n=H_n(q)\).
Applying \cref{lem:sx} to \eqref{eq:elliptic-q}, we obtain $H_n(q)=H_{n-1}(Q)$ and the series $Q(x)$ as follows:
\begin{align*}
Q(x)=\frac{ab-L+2d+1+(b^4+ab-L+2d+1)x}{1+(ab+2d+2)x+(2ab-L+4d+2)x^2+x^2(-1-(ab+2d+1)x)Q(x)}.
\end{align*}
The series \(Q(x)\) satisfies the Wang--Zhang condition with
\begin{align*}
 a_0&=-\delta,&
 b_0&=-\delta+b^4,&
 c_0&=ab+2d+2,\\
 d_0&=-\delta+ab+2d+1,&
 f_0&=-(ab+2d+1).
\end{align*}
According to \eqref{eq:wz-alpha}--\eqref{eq:wz-beta}, we obtain 
\begin{equation}\label{eq:elliptic-unscaled-parameters}
 \alpha_Q=b^8,\qquad \beta_Q=b^8\delta.
\end{equation}
The Sulanke--Xin identity (\cref{lem:sx}) has leading coefficient one in this application, so
\[h_n=H_n(q)=H_{n-1}(Q)\qquad(n\geq1).
\]
It follows from \eqref{eq:elliptic-unscaled-parameters} that \((h_n)\) obeys the \((b^8,b^8\delta)\) identity for \(n\geq5\).  Direct substitution of the initial determinants verifies the remaining identity at \(n=4\).  More
explicitly, set
\[\Theta=a^2b^2d-ab^3c+3abd^2+b^4-2b^2cd+2d^3.
\]
The first three determinants of \(Q\) are
\[ H_1(Q)=-\delta,\qquad H_2(Q)=-b^4\Theta,\qquad
 H_3(Q)=-b^{12}\Theta+b^8\delta^3.
\]
Consequently
\[(h_0,h_1,h_2,h_3,h_4)
 =(1,1,-\delta,-b^4\Theta,-b^{12}\Theta+b^8\delta^3),
\]
and the \(n=4\) identity is immediate.

It remains to transport the normalization.  If \(T_n=b^{-n^2+2n}h_n\), then comparison of the exponents in the three products in
\eqref{eq:somos4-bilinear} gives
\[ \alpha_T=b^{-6}\alpha_Q=b^2,\qquad \beta_T=b^{-8}\beta_Q=\delta.
\]
The first four terms, calculated from \(g(x)\) are
\[ 1,\quad b,\quad -\delta,\quad -b\bigl(a^2b^2d+ab(3d^2-b^2c)+b^4-2b^2cd+2d^3\bigr),
\]
with the sign convention of \cite{Barry2023Elliptic}. 
This completes the proof.
\end{proof}

We next give the definitions of general elliptic curves $E$ and their division polynomials. We only introduce the definitions and formulas that are needed; for further details, see~\cite{Shipsey,LW. SSBCHMW,Ward1948,WardDuck}.

Let \(E\) be given by the Weierstrass equation
$$E : y^2 + a_1xy + a_3y = x^3 + a_2x^2 + a_4x + a_6,$$
and let 
\begin{align*}
b_2 = a_1^2 + 4a_2,\quad b_4 = a_1 a_3 + 2a_4, \quad b_6 = a_3^2 + 4a_6, 
\quad b_8 = a_1^2 a_6 + 4a_2 a_6 - a_1 a_3 a_4 + a_2 a_3^2 - a_4^2.
\end{align*}
Note that $4b_8=b_2b_6-b_4^2$.
When \(\text{char}(K) \neq 2, 3\), one may instead use the short Weierstrass form 
\(E : y^2 = x^3 + Ax + B\).

We define \emph{division polynomials} \(\psi_m \in \mathbb{Z}[a_1, \dots, a_6, x, y]\) using initial values
\begin{align*}
&\psi_0 = 0,\quad \psi_1 = 1,\quad \psi_2 = 2y + a_1x + a_3,
\quad \psi_3 = 3x^4 + b_2x^3 + 3b_4x^2 + 3b_6x + b_8,
\\ & \psi_4 = \psi_2 \cdot \bigl(2x^6 + b_2x^5 + 5b_4x^4 + 10b_6x^3 + 10b_8x^2 + (b_2b_8 - b_4b_6)x + (b_4b_8 - b_6^2)\bigr),
\end{align*}
and then inductively by the formulas
\begin{align*}
\psi_{2m+1} &= \psi_{m+2}\psi_m^3 - \psi_{m-1}\psi_{m+1}^3 \quad \text{for } m \geq 2,
\\ \psi_{2}\psi_{2m} &= \psi_{m-1}^2\psi_m\psi_{m+2} - \psi_{m-2}\psi_m\psi_{m+1}^2 \quad \text{for } m \geq 3.
\end{align*}
We can verify that \(\psi_m\) is a polynomial for all \(m \geq 1\).
In general, the division polynomials satisfy the following equality:
\begin{align}\label{Equation-Psi-Divis}
\psi_{m+n} \psi_{m-n} =\psi_{m+1} \psi_{m-1} \psi_{n}^2-\psi_{n+1}\psi_{n-1} \psi_{m}^2.    
\end{align}

\begin{theorem}{\em (Conjectured in \cite[Conjecture~2]{Barry2023Elliptic})}\label{thm:elliptic-Division-2023}
Notation follows \cref{thm:elliptic-2023}. Assume that \(b\ne0\) and \(P=(x,y)=(0,0)\).
Let \(E\) be given by the Weierstrass equation $E: y^2+axy+by=x^3+cx^2+dx$.
Then we have $S_n=\psi_{n+1}(P)$ for $n\geq0$.
In particular, the determinants $S_n=b^{-n^2+2n}h_n$ give the shifted division-polynomial sequence at the point \((0,0)\).
\end{theorem}
\begin{proof}
From equations for $b_i$ ($i=2,4,6,8$) and $\psi_{i}$ ($i=1,2,3,4$) above, we can compute the following:
$$\psi_{1}(P)=1,\quad \psi_{2}(P)=b,\quad \psi_{3}(P)=-\delta,\quad \psi_{4}(P)=-b\bigl(a^2b^2d+ab(3d^2-b^2c)+b^4-2b^2cd+2d^3\bigr).$$
Therefore, $S_n$ and $\psi_{n+1}(P)$ satisfy the same initial conditions.
Since $\psi_{n}(P)$ satisfies \eqref{Equation-Psi-Divis}, by setting $n=2$, we obtain
$\psi_{m+2}\psi_{m-2} = \psi_{m+1}\psi_{m-1}\psi_2^2 - \psi_3\psi_1\psi_m^2$.
Furthermore, we have
$$\psi_{m+2}(P)\psi_{m-2}(P) = b^2 \psi_{m+1}(P)\psi_{m-1}(P) + \delta \psi_m(P)^2.$$
Setting $m=n-1$, we get
$$\psi_{n+1}(P)\psi_{n-3}(P) = b^2 \psi_n(P)\psi_{n-2}(P) + \delta \psi_{n-1}(P)^2.$$
Consequently, both $(\psi_{n+1}(P))_{n\geq 0}$ and $(S_n)_{n\geq 0}$ satisfy the $(b^2, \delta)$ Somos-4 equation.
Suppose inductively that they agree through index \(n-1\), where \(n\geq4\). Subtracting the two recurrences gives $(S_n-\psi_{n+1}(P))\psi_{n-3}(P)=0$.
It remains only to prove that $\psi_{m}(P)\neq 0$ for all $m\geq 1$.
This requires more knowledge of elliptic curves; we omit it here. 
The interested reader may fill in the details.
(We know that $\psi_i(P) = 0 \iff iP = \mathcal{O}$ for some $i$, where $\mathcal{O}$ denotes the point at infinity. By specializing \((a,b,c,d)=(0,1,0,-1)\), it can be proved that $P$ has infinite order.)
This completes the proof.
\end{proof}

\section{Quadratic and cubic continued fractions}\label{Section-Tree}

This section mainly concerns Hankel determinants corresponding to two different kinds of continued fractions.
The following theorem concerns quadratic continued fractions.

\begin{theorem}\label{thm:quadratic-master}
If the generating function $F(x)$ satisfies 
\begin{equation}\label{eq:quadratic-master}
 F(x)=\frac1{1-rx-sx^2-tx^2F(x)},
\end{equation}
then \((H_n(F))_{n\geq0}\) is an \((\alpha,\beta)\) Somos-4 sequence, where
\begin{align*}
 \alpha=r^2t^2,\qquad \beta=t^2\bigl((s+t)^2-r^2t\bigr).
\end{align*}
\end{theorem}
\begin{proof}
Set \(Q(x)=tF(x)\). Then we have
\[Q(x)=\frac{t}{1-rx-sx^2-x^2Q(x)},
\]
which is \eqref{eq:wz-form} with
\[ a_0=t,\quad b_0=0,\quad c_0=-r,\quad d_0=-s,\quad f_0=f_1=0,
 \quad a_1=s+t.
\]
Equations \eqref{eq:wz-alpha}--\eqref{eq:wz-beta} give
\[
 \alpha=r^2t^2,\qquad
 \beta=-r^2t^3+(s+t)^2t^2.
\]
Finally, \(H_n(Q)=t^nH_n(F)\), and multiplication by \(t^n\) leaves both Somos parameters unchanged.  The singular cases follow from \cref{rem:specialization}.
\end{proof}

\begin{corollary}{\em (Conjectured in \cite[Conjecture~10]{Barry2010} and proved in \cite[Theorem 1.2]{ChangHu2012})}\label{cor:barry2010}
Let \(a_0=1\), \(a_1=r\), and
\begin{align}\label{eq:RST-Conject}
a_n=ra_{n-1}+sa_{n-2}+t\sum_{j=0}^{n-2}a_ja_{n-2-j}\qquad(n\geq2).
\end{align}
Let \(F(x)=\sum_{n\geq0}a_nx^n\). Then \((H_n(F))_{n\geq0}\) forms an $\bigl(r^2t^2,\;t^2((s+t)^2-r^2t)$ Somos-4 sequence.  
\end{corollary}
\begin{proof}
The recurrence translates to $F(x)=1+rxF(x)+sx^2F(x)+tx^2F(x)^2$,
which is equivalent to \eqref{eq:quadratic-master}; the assertion therefore follows from \cref{thm:quadratic-master}.
\end{proof}

More explicitly, \cref{thm:quadratic-master} directly implies the following conjectures.

\begin{corollary}{\em (Conjectured in \cite[Conjecture~5]{Barry2011Invariant})}\label{Corollary-Conjecture-5-Bary}
Let \(F_{a,b}(0)=1\) be defined by
\[F_{a,b}=\frac{1}{1-ax-(b-1)x^2-x^2F_{a,b}}.\]
Then \(\bigl(H_{n+1}(F_{a,b})\bigr)_{n\geq0}\) is an \((a^2,b^2-a^2)\) Somos-4 sequence.
\end{corollary}
\begin{proof}
The proof follows by setting \((r,s,t)=(a,b-1,1)\) in \cref{thm:quadratic-master}.
\end{proof}

\begin{corollary}{\em (Conjectured in \cite[Section~7, Page~227]{RajkovicBarrySavic2012})}\label{Corollary-RBS-111}
For \((r,s,t)=(1,1,1)\) in \eqref{eq:RST-Conject}, the coefficient sequence is
[A128720] on OEIS \cite{Sloane23} and begins
\[1,1,3,6,16,40,109,297,\ldots,\]
and its Hankel transform begins
\[1,1,2,5,17,109,706,9529,149057,\ldots.\]
Hence the latter is the \((1,3)\) Somos-4 sequence [A174168] on OEIS \cite{Sloane23}.
\end{corollary}
\begin{proof}
The proof follows by setting \((r,s,t)=(1,1,1)\) in \cref{thm:quadratic-master}.
The displayed initial determinants agree with [A174168]; together with the proved \((1,3)\) recurrence, this gives the asserted identification.
\end{proof}

A Motzkin-Schr\"oder path is a lattice path from $(0,0)$ to $(n,0)$ using four types of steps, $U=(1,1)$, $D=(1,-1)$, $E=(1,0)$ and $\overline{N}=(2,0)$, that does not go below the axis $y=0$.
Such paths are widely studied; see, for example, \cite{ChenYanYang,Deutsch2002,Ramirez}.

\begin{corollary}{\em (Conjectured in \cite[Conjecture~8]{Barry2023Motzkin})}\label{Conjecture-8-Barry-ARST}
Let \(A_n(r,s,t)\) count Motzkin--Schr\"oder paths with steps $U=(1,1)$, $D=(1,-1)$, $E=(1,0)$, $\overline{N}=(2,0)$.
Give \((E,\overline{N},U)\) \((r,s,t)\) colors, and leave \(D\) uncolored respectively.
Then \(\left(H_{n+1}\bigl(A(r,s,t)\bigr)\right)_{n\geq 0}\) is an $\bigl((rt)^2,\;t^2((s+t)^2-r^2t)\bigr)$ Somos-4 sequence. 
\end{corollary}
\begin{proof}
According to \cite[Proposition~1]{Barry2023Motzkin}, the number $A_n=A_n(r,s,t)$ is given by the following recurrence
$$A_n=rA_{n-1}+sA_{n-2}+t\sum_{i=0}^{n-2}A_iA_{n-2-i},\qquad A_0=1,\ A_1=r.$$
The proof follows immediately from \cref{cor:barry2010}.
\end{proof}

A $3$-generalized Motzkin path is a lattice path with unit up and down steps given, respectively, by $(1,1)$ and $(1,-1)$, and three horizontal steps of lengths, respectively, of $1$, $2$, and $3$.
The generating function $G(x)$ of these $3$-generalized Motzkin paths is given by the continued fraction expression
$$G(x)=\frac{1}{1-x-x^2-x^3-x^2G(x)}.$$
We prove, with the aid of $G(x)$, that not all quadratic generating functions have Hankel transforms which are $(\alpha,\beta)$ Somos-4 sequences.

\begin{proposition}\label{prop:motzkin-blanket-false}
Let $G(x)=\sum_{n\geq 0}g_nx^n$ be the generating function for the 3-generalized Motzkin paths.  Its Hankel transform is not an
\((\alpha,\beta)\) Somos-4 sequence for any constants \(\alpha,\beta\).
\end{proposition}
\begin{proof}
If \((H_n(G))_{n\geq0}\) were a Somos--4 sequence, then its forward shift
\((H_{n+1}(G))_{n\geq0}\) would be one as well by
\cref{lem:somos-invariance}.  It therefore suffices to disprove the
recurrence for the shifted sequence.
Exact expansion gives
\[
 (g_n)_{n\geq0}=1,1,3,7,18,48,132,372,\ldots
\]
and
\[
 (H_{n+1}(G))_{n\geq0}=1,2,2,-5,-17,-29,-70,-70,169,577,\ldots.
\]
The identities at indices \(4\) and \(5\) would force $\alpha=\frac{193}{114}$ and $\beta=-\frac{1}{57}$.
At index \(6\), however, the difference between the left- and right-hand sides of \eqref{eq:somos4-bilinear} is
\(-43367/114\), a contradiction.
\end{proof}

The following theorem concerns cubic continued fractions.

\begin{theorem}{\em (Conjectured in \cite[Conjecture~26]{Barry2011Narayana})}\label{thm:cubic-master}
If the generating function $F(x)$ satisfies 
\begin{align*}
 F(x)=\frac1{1-rx-sx^2-tx^3F(x)},
\end{align*}
then the shifted Hankel sequence \((H_{n+1}(F))_{n\geq0}\) is a $(t^2,-st^2)$ Somos-4 sequence.
\end{theorem}
\begin{proof}
Apply \cref{lem:sx} with $(a,b,c,d,e,f)=(1,0,-r,-s,0,-t)$.
It gives \(H_{n+1}(F)=H_n(G)\), where
\[G(x)=\frac{s+(t-rs)x}{1-rx+sx^2-x^2G(x)}.
\]
Thus \(G(x)\) satisfies the Wang--Zhang condition with
\[
 a_0=s,\quad b_0=t-rs,\quad c_0=-r,\quad d_0=s,\quad f_0=0.
\]
A direct calculation gives
\[
 f_1=r-\frac ts,\qquad
 a_1=\frac{rt}{s}-\frac{t^2}{s^2},\qquad
 c_0+f_0+f_1=-\frac ts.
\]
Moreover \((f_0-f_1)(c_0+f_0+f_1)=a_1\).  Hence \cref{thm:wz} gives \(\alpha=t^2\) and \(\beta=-st^2\).  Specialization covers
\(s=0\). This completes the proof.
\end{proof}

\begin{corollary}{\em (Conjectured in \cite[Page 27]{Barry2021Catalan})}\label{cor:2021-p27}
Let $F(x)$ be the generating function given by
\[ F(x)=\frac1{1-x-\alpha x^2} \mathcal{C}\!\left(\frac{x^3}{(1-x-\alpha x^2)^2}\right).
\]
Then the Hankel transform \((H_{n}(F))_{n\geq 0}\) is a \((1,-\alpha)\) Somos-4 sequence.
\end{corollary}
\begin{proof}
By the definition of $F(x)$ and \(\mathcal{C}(x)=1+x\mathcal{C}(x)^2\), we have 
\[ F(x)=\frac1{1-x-\alpha x^2-x^3F(x)}.\]
\cref{thm:cubic-master} gives the recurrence for \((H_{n+1}(F))_{n\geq0}\), hence for \((H_n(F))_{n\geq0}\) at every
\(n\geq5\).  Direct expansion gives
\[
 H_0(F)=H_1(F)=1,\quad H_2(F)=\alpha,\quad
 H_3(F)=\alpha-1,\quad H_4(F)=-\alpha^3+\alpha-1,
\]
which verifies the remaining identity at \(n=4\).
\end{proof}

\begin{corollary}{\em (Conjectured in \cite[Example 14]{Barry2023Motzkin})}\label{thm:new-motzkin-9-5}
Let $F(x)$ be the generating function given by
\[
 F(x)=\frac{1}{x}\cdot \left(\mathcal{C}\!\left(\dfrac{x}{1-x-x^2}\right)-1\right).
\]
Then \(\bigl(H_{n+1}(F)\bigr)_{n\geq0}\) is a \((9,-5)\) Somos-4 sequence. 
\end{corollary}
\begin{proof}
The Catalan equation \(\mathcal{C}(z)=1+z\mathcal{C}(z)^2\) gives
\[F(x)=1+3xF(x)+x^2F(x)+x^2F(x)^2=\frac1{1-3x-x^2-x^2F(x)}.
\]
Since this equation has a unique solution with constant term \(1\), this is precisely the specialization \((r,s,t)=(3,1,1)\) of 
\cref{cor:barry2010}.  The parameters there are $\alpha=(rt)^2=9$ and $\beta=t^2\bigl((s+t)^2-r^2t\bigr)=-5$.
\end{proof}

\section{Hankel determinants of generalized Bernstein arrays}\label{sec:bernstein}

We first recall the common combinatorial concept of a Riordan array, which was first introduced in~\cite{Shapiro1991}.
Let \(g(x),f(x)\in\Ring[[x]]\), \(f(0)=0\), and \(f'(0)\) is a unit.  The \emph{Riordan array} \((g(x),f(x))\) is the infinite lower-triangular matrix whose \((n,k)\)-entry is
\[r_{n,k}=[x^n]g(x)f(x)^k,\qquad n,k\ge0.\]
Furthermore, if \(g(0)\) is a unit, then \((g(x),f(x))\) is called a \emph{proper Riordan array}.
It is clear that a proper Riordan array is invertible.
The series $f(x)$ is compositionally invertible, with inverse $\bar{f}(x)$, where $f(\bar{f}(x))=x$ and $\bar{f}(f(x))=x$.
For more results on the Riordan group and applications, see \cite{LW. SSBCHMW,Cheon2017,Luzon2010,LuzonMoron2008}.

The $(n+1)$ \emph{Bernstein polynomials} of degree $n$ are the polynomials $B_{n,k}=\binom{n}{k}s^k(1-s)^{n-k}$.
We call the coefficient array $B_{n,k}(s)$ the \emph{Bernstein array}.
This is a lower triangular array.
Barry \cite{Barry2012Bernstein} proved that the Bernstein array $B_{n,k}$ is given by the Riordan array product
$$\left(\frac{1}{1-(1-s)x},\frac{sx}{1-(1-s)x}\right)=\left(\frac{1}{1-x},\frac{x}{1-x}\right) \cdot (1,sx)\cdot \left(\frac{1}{1-x},\frac{x}{1-x}\right)^{-1}.$$

Let $(g(x),f(x))$ be a proper Riordan array. Barry defines the \emph{generalized Bernstein array} as 
$$\mathbb{B}^{(g,f)}=(g(x),f(x))\cdot (1,sx)\cdot (g(x),f(x))^{-1}.$$
Let \(\bar f(x)\) be the compositional inverse of \(f(x)\). Then the first column \cite{Barry2012Bernstein} of its generalized Bernstein array has generating function
\begin{equation}\label{eq:bernstein-first-column}
 B_s(x)=\frac{g(x)}{g(\bar f(sf(x)))}.
\end{equation}

Next, we will prove various $(\alpha,\beta)$ Somos-4 conjectures involving the Hankel transform of $B_s(x)$.

\begin{theorem}{\em (Conjectured in \cite[Section 4]{Barry2012Bernstein})}\label{thm:bernstein-catalan}
If the Bernstein array $\mathbb{B}^{(g,f)}$ is generated by the Riordan array $(g(x),f(x))=(1-x,x(1-x))$,
then the Hankel transform \(H_{n+1}(B_s)\) is a
\[
 \bigl(4s^4(s-1)^2,\;-s^4(s-1)^3(3s+1)\bigr)
\]
Somos-4 sequence.
\end{theorem}
\begin{proof}
For \((g(x),f(x))=(1-x,x(1-x))\), the series in \eqref{eq:bernstein-first-column} \cite[Section 4]{Barry2012Bernstein} is given by
\[ B_s(x)=\frac{1-\sqrt{1-4sx(1-x)}}{2sx},
\]
and hence
\begin{equation}\label{eq:bernstein-catalan-quadratic}
 B_s(x)=1-x+sxB_s(x)^2.
\end{equation}
Now we consider using \cref{lem:jacobi-step} with $u=s-1$ and $v=s^2-1$.
Write
\[B_s(x)=\frac1{1-(s-1)x-(s^2-1)x^2q(x)}.
\]
Substitution in \eqref{eq:bernstein-catalan-quadratic} followed by cancellation gives
\[(s^2-1)x^2(x-1)q(x)^2+\bigl(1-2sx+2(s-1)x^2\bigr)q(x)+\frac{s-1}{s+1}x-1=0.
\]
Thus \(Q(x)=(s^2-1)q(x)\) satisfies
\[
 Q(x)=\frac{(s^2-1)-(s-1)^2x}{1-2sx+2(s-1)x^2+x^2(-1+x)Q(x)}.
\]
In \cref{thm:wz}, this is \eqref{eq:wz-form} with
\[a_0=s^2-1,\quad b_0=-(s-1)^2,\quad c_0=-2s,\quad d_0=2(s-1),\quad f_0=1.
\]
The quantities in \cref{thm:wz} simplify to
\[f_1=\frac{s-1}{s+1},\qquad a_1=\frac{s^2(s-1)(s+3)}{(s+1)^2},\qquad c_0+f_0+f_1=-\frac{2s^2}{s+1}.
\]
Equations \eqref{eq:wz-alpha}--\eqref{eq:wz-beta} now give the claimed parameters $\alpha$ and $\beta$. 
By \cref{lem:jacobi-step}, we have
\[
 H_{n+1}(B_s)=(s^2-1)^nH_n(q)=H_n(Q),
\]
This completes the proof.
\end{proof}

Note that the conjecture in the theorem below, as originally stated, contains a typo. We state and prove the correct version.

\begin{theorem}{\em (Conjectured in \cite[Conjecture 16]{Barry2012Bernstein})}\label{thm:bernstein-c16}
If the Bernstein array $\mathbb{B}^{(g,f)}$ is generated by the Riordan array $$(g(x),f(x))=\left(\frac1{1+ax},\frac{x(1+bx)}{1+ax}\right),$$
then the Hankel transform \(H_{n+1}(B_s)\) is an \((\alpha,\beta)\) Somos-4 sequence, where
\begin{align}
\alpha={}&b^2s^4(s-1)^2(a-b)^2(a-2b)^2,\label{eq:c16-alpha}\\
\beta={}&b^3s^4(s-1)^2(a-b)^3
 \Bigl(a^2s^2+b(b-a)(3s^2-2s-1)\Bigr).
\label{eq:c16-beta}
\end{align}
\end{theorem}
\begin{proof}
For convenience, we set \(D=1+ax\) and \(E=1+bx\). 
According to $f(\bar{f}(t))=t$, we have $f(\bar f(sf(x)))=sf(x)$.
If \(y(x)=\bar f(sf(x))\), then we have $f(y(x))=sf(x)$.
Therefore, we obtain
$$\frac{y(x)(1+by(x))}{1+ay(x)}=sf(x).$$
Thus we get
\[by(x)^2+(1-asf(x))y(x)-sf(x)=0,\qquad B_s(x)=\frac{1+ay(x)}{1+ax}.
\]
Eliminating \(y(x)\) yields the quadratic equation
\begin{equation}\label{eq:bernstein-linear-quadratic}
 bD^2B_s(x)^2+\bigl((a-2b)D-a^2sxE\bigr)B_s(x)+b-a=0.
\end{equation}
Assume that the power series expansion of $B_s(x)$ is
$$B_s(x)=B_0+B_1x+B_2x^2+\cdots.$$
According to \eqref{eq:bernstein-linear-quadratic}, we have
\begin{align*}
&b(1+ax)^2(B_0+B_1x+B_2x^2+O(x^3))
\\ &\qquad \qquad+\bigl((a-2b)(1+ax)-a^2sx(1+bx)\bigr)(B_0+B_1x+B_2x^2+O(x^3))+b-a=0.
\end{align*}
It follows from the above equation that
$$B_0=[x^0]B_s(x)=1, \qquad B_1=[x]B_s(x)=a(s-1), \qquad B_2=a(s-1)(a(s-1)-bs).$$
Therefore, we have
$$H_2(B_s)=a\,b\,s(1-s).$$

Now we set \(\eta=a\,b\,s(1-s)\). We also set
\[ B_s(x)=\frac1{1-a(s-1)x-\eta x^2q(x)}.
\]
On substituting in \eqref{eq:bernstein-linear-quadratic}, the remaining equation is
\begin{align*}
b(a-b)s(s-1)x^2q(x)^2+\bigl(1+(a(1-s)+2bs)x+absx^2\bigr)q(x)-(1+ax)=0.
\end{align*}
Let \(A=b(a-b)s(s-1)\) and \(Q(x)=-Aq(x)\).  Then we have 
\[Q(x)=\frac{-A-Aax}{1+(a(1-s)+2bs)x+absx^2-x^2Q(x)}.
\]
This satisfies the Wang--Zhang condition with
\[a_0=-A,\quad b_0=-Aa,\quad c_0=a(1-s)+2bs,\quad d_0=a\,b\,s,\quad f_0=0.
\]
In particular
\[f_1=-a,\qquad a_1=-s(a-b)(a+bs-b),\qquad c_0+f_0+f_1=-s(a-2b).
\]
Substitution into \cref{thm:wz} gives \eqref{eq:c16-alpha} and \eqref{eq:c16-beta}.  Finally, we obtain
\[
 H_{n+1}(B_s)=\eta^nH_n(q)=\left(\frac{a}{a-b}\right)^nH_n(Q).
\]
The extra factor is geometric in \(n\), so it does not change the Somos parameters. 
Degenerate parameter values follow by specialization.
\end{proof}

\begin{corollary}{\em (Conjectured in \cite[Conjecture 15]{Barry2012Bernstein})}\label{cor:bernstein-c15}
If the Bernstein array $\mathbb{B}^{(g,f)}$ is generated by the Riordan array $$(g(x),f(x))=\left(\frac1{1-x},\frac{x(1+rx)}{1-x}\right),$$
then the Hankel transform \(H_{n+1}(B_s)\) is an \((\alpha,\beta)\) Somos-4 sequence, where
\begin{align*}
\alpha&=r^2s^4(s-1)^2(r+1)^2(2r+1)^2,\\
\beta&=-s^4(s-1)^2\bigl(r(r+1)\bigr)^3\Bigl(r(r+1)(3s^2-2s-1)+s^2\Bigr).
\end{align*}
\end{corollary}
\begin{proof}
The proof is completed by taking \(a=-1\) and \(b=r\) in \cref{thm:bernstein-c16}.
\end{proof}

The printed value of \(\beta\) in \cite[Conjecture 16]{Barry2012Bernstein} has a minus sign in front and contains \(b-ab\) inside the last parenthesis.  Both are incompatible with \cite[Conjecture 15]{Barry2012Bernstein}.  For example, at \((a,b,s)=(-1,1,2)\), the actual determinants give \(\beta=-2304\), whereas the printed formula gives \(+2304\).
The corrected expression is \eqref{eq:c16-beta}.

\begin{corollary}{\em (Conjectured in \cite[Page 12]{Barry2012Bernstein})}\label{thm:bernstein-c171}
If the Bernstein array $\mathbb{B}^{(g,f)}$ is generated by the Riordan array $$(g(x),f(x))=\left(\frac1{1-x},\frac{x(1+x)}{1-x}\right),$$
then the Hankel transform \(H_{n+1}(B_s)\) is an \((36s^4(s-1)^2,-8s^4(s-1)^2(7s^2-4s-2))\) Somos-4 sequence.
\end{corollary}
\begin{proof}
The assertion is the specialization \(r=1\) in \cref{cor:bernstein-c15}.
\end{proof}

\begin{theorem}{\em (Conjectured in \cite[Conjecture 17]{Barry2012Bernstein})}\label{thm:bernstein-c172}
If the Bernstein array $\mathbb{B}^{(g,f)}$ is generated by the Riordan array $$(g(x),f(x))=\left(\frac1{1+ax+bx^2},\frac{x}{1+ax+bx^2}\right),$$
then the Hankel transform of $B_s(x)$ is an \((\alpha,\beta)\) Somos-4 sequence, where
\begin{align*}
\alpha&=a^2b^2s^4(s-1)^2,\\
\beta&=-b^3s^4(s-1)^2\bigl(s^2(a^2-b)-2bs-b\bigr).
\end{align*}
\end{theorem}
\begin{proof}
Let \(P(x)=1+ax+bx^2\). Then \(g(x)=1/P(x)\) and \(f(x)=x/P(x)\). By \eqref{eq:bernstein-first-column}, we have
\[B_s(x)=\frac{g(x)}{g(\bar f(sf(x)))}.
\]
Set \(y=\bar f(sf(x))\). Since \(\bar f\) is the compositional inverse of \(f\), we have \(f(y)=s f(x)\), i.e.,
\[\frac{y}{P(y)}=s\frac{x}{P(x)}.
\]
Cross-multiplying gives \(yP(x)=sxP(y)\), hence $y=sx\frac{P(y)}{P(x)}$.
But from 
$$B_s(x)=\frac{g(x)}{g(y)}=\frac{P(y)}{P(x)},$$ we get \(y=sxB_s(x)\). On the other hand, \(B_s(x)=P(y)/P(x)\) implies \(B_s(x)P(x)=P(y)\). Substituting \(y=sxB_s(x)\) yields
\[
B_s(x)(1+ax+bx^2)=1+a(sxB_s(x))+b(sxB_s(x))^2.
\]
By expanding and rearranging the above equation, we obtain
\[
B_s(x)+axB_s(x)+bx^2B_s(x)=1+asxB_s(x)+bs^2x^2B_s(x)^2.
\]
Thus
\[
1=B_s(x)\bigl(1+a(1-s)x+bx^2-bs^2x^2B_s(x)\bigr).
\]
The first-column generating function in \eqref{eq:bernstein-first-column} is the
unique solution of
\begin{equation}\label{eq:bernstein-quadratic-denominator}
 B_s(x)=\frac1{1+a(1-s)x+bx^2-bs^2x^2B_s(x)}.
\end{equation}
Set \(Q(x)=bs^2B_s(x)\).  Equation\eqref{eq:bernstein-quadratic-denominator} becomes
\[Q(x)=\frac{bs^2}{1+a(1-s)x+bx^2-x^2Q(x)}.
\]
According to \cref{thm:wz}, we obtain
\[a_0=bs^2,\quad b_0=0,\quad c_0=a(1-s),\quad d_0=b,\quad f_0=f_1=0,\quad a_1=b(s^2-1).
\]
The result follows immediately from \cref{thm:wz}; the factor \(H_n(Q)=(bs^2)^nH_n(B_s)\) is geometric.
This completes the proof.
\end{proof}

\begin{corollary}{\em (Conjectured in \cite[Page 11]{Barry2012Bernstein})}\label{thm:bernstein-c173}
If the Bernstein array $\mathbb{B}^{(g,f)}$ is generated by the Riordan array $$(g(x),f(x))=\left(\frac1{1+x+x^2},\frac{x}{1+x+x^2}\right),$$
then the Hankel transform of $B_s(x)$ is an \((s^4(s-1)^2,s^4(s-1)^2(2s+1))\) Somos-4 sequence.
\end{corollary}
\begin{proof}
The assertion is the specialization \(a=b=1\) in \cref{thm:bernstein-c172}.
\end{proof}

\section{A-matrix Riordan arrays}\label{Section-Five}

Recall that a lower triangular array $(t_{n,k})_{0\leq n,k\leq \infty}$. It is a Riordan array if and only if there exists a sequence $a_n$, $n\geq 0$, such that 
$$t_{n,k}=\sum_{i=0}^{\infty} a_i t_{n-1,k-1+i}.$$
Because the matrix is lower triangular, this sum is actually a finite sum.

In fact, we have the following necessary and sufficient condition \cite{HeTianxiao,Barry2019AMatrix}.
A lower-triangular array \( (t_{n,k})_{0 \leq n,k \leq \infty} \) is a Riordan array if and only if there exists another array \( A = (a_{i,j})_{i,j \in \mathbb{N}} \) with \( a_{0,0} \neq 0 \), and a sequence \( (\rho_j)_{j \in \mathbb{N}} \) such that
\[
t_{n+1,k+1} = \sum_{i \geq 0} \sum_{j \geq 0} a_{i,j} t_{n-i,k+j} + \sum_{j \geq 0} \rho_j t_{n+1,k+j+2}.
\]
The matrix $A$ is called the $A$-matrix.
For a given Riordan array $(g(x),f(x))$, the relationship between $f(x)$ and the pair $(A,\rho)$ is 
\begin{align}\label{Equation-F-A-RHO}
\frac{f(x)}{x}=\sum_{i\geq 0}x^i R^{(i)}(f(x))+\frac{f(x)^2}{x}\rho(f(x)),
\end{align}
where $R^{(i)}$ is the generating series of the $i$-th row of $A$, and $\rho(x)$ is the generating series of the sequence $\rho_n$.

Next, for a given pair $(A,\rho)$, we consider the Hankel transform of $f(x)/x$ given by the associated Equation \eqref{Equation-F-A-RHO}.

\begin{theorem}{\em (Conjectured in \cite[Conjecture 6]{Barry2019AMatrix})}\label{thm:amatrix-c6}
Let\[
 A=\begin{pmatrix}1&a&b\\1&c&d\end{pmatrix} \quad\text{and}\quad \rho_n=0\quad \text{for all}\quad n.
\]
The Hankel transform of $f(x)/x$  associated with Equation \eqref{Equation-F-A-RHO} is an \((\alpha,\beta)\) Somos-4 sequence with
\begin{align*}
\alpha={}&(b+ab+d)^2\\
\beta={}&b^4-b^3(2+3a+a^2-2c)+b(a+a^2-ac-2d)d+(1+a-c)d^2
\\ &\quad -b^2(c+ac-c^2+2d+3ad).
\end{align*}
\end{theorem}
\begin{proof}
To find $f(x)$, we solve the equation
$$\frac{f(x)}{x}=1+af(x)+bf(x)^2+x(1+cf(x)+df(x)^2).$$
Then we obtain
\[
f(x) = \frac{1 - ax - cx^2 - \sqrt{1 - 2ax - (a^2 - 2(2b + c))x^2 + 2(ac - 2(b + d))x^3 + (c^2 - 4d)x^4}}{2x(dx + b)}.
\]
It follows that
\[
G(x)=\frac{f(x)}{x} = \frac{1 + x}{1 - ax - cx^2} \mathcal{C} \left( \frac{x^2(1 + x)(b + dx)}{(1 - ax - cx^2)^2} \right),
\]
where \(\mathcal{C}(x)\) is the Catalan generating function.
Using the identity \(\mathcal{C}(x)=1+x\mathcal{C}(x)^2\), we have 
\begin{equation}\label{eq:amatrix-c6-quadratic}
 G(x)=\frac{1+x}{1-ax-cx^2-x^2(b+dx)G(x)}.
\end{equation}
Set \(Q(x)=bG(x)\). Equation \eqref{eq:amatrix-c6-quadratic} becomes
\[
 Q(x)=\frac{b+bx}{1-ax-cx^2+x^2(-1-(d/b)x)Q(x)}.
\]
Thus, in \cref{thm:wz}, we get
\[a_0=b,\quad b_0=b,\quad c_0=-a,\quad d_0=-c,\quad f_0=-\frac db.
\]
The auxiliary quantities are
\[ f_1=-1,\qquad a_1=-a+b+c-1,\qquad c_0+f_0+f_1=-\frac{ab+b+d}{b}.
\]
Substitution in \cref{thm:wz} gives the displayed parameters after expansion. Since \(H_n(Q)=b^nH_n(G)\), the parameters are unchanged. This completes the proof.
\end{proof}

For the specialization \((a,b,c,d)=(0,1,r,1)\), \cref{thm:amatrix-c6} gives \((\alpha,\beta)=(4,r^2-4)\).

We next turn to the notation \(\rho_n=0^n\).  It means \(\rho_0=1\) and \(\rho_n=0\) for \(n>0\), rather than the identically zero sequence.  

\begin{theorem}{\em (Conjectured in \cite[Conjecture 14]{Barry2019AMatrix})}\label{thm:amatrix-c14}
Let\[
 A=\begin{pmatrix}1&A&B\\1&C&D\end{pmatrix} \quad\text{and}\quad \rho_n=0^n\quad \text{for all}\quad n.
\]
The shifted Hankel transform of $(H_{n+1}(f(x)/x))_{n\geq 0}$ associated with Equation \eqref{Equation-F-A-RHO} is an \((\alpha,\beta)\) Somos-4 sequence with
\begin{align*}
\alpha={}&\bigl(4+A^2+3B+A(4+B)+C+D\bigr)^2\\
\beta={}&-16-A^5+B^4-3A^4(3+B)+2B^3(-2+C)-8C-4C^2-C^3\\
&+B^2(-28-C+C^2-8D)-8D-6CD-2C^2D-D^2-CD^2\\
&-A^3(32+23B+3B^2+C+2D)-2B\bigl(20+C^2+9D+D^2+3C(2+D)\bigr)\\
&-A^2\bigl(56+18B^2+B^3+10D+C(6+D)+B(66+C+5D)\bigr)\\
&-A\bigl(48+3B^3+2C^2+16D+D^2+B^2(38-2C+3D)\\
&\qquad +C(12+5D)+B(84-C^2+19D+C(8+D))\bigr).
\end{align*}
\end{theorem}
\begin{proof}
By solving the equation
\[
\frac{f(x)}{x} = 1 + Af(x) + Bf(x)^2 + x(1 + Cf(x) + Df(x)^2) + \rho_0 \frac{f(x)^2}{x},
\]
we find that the generating function is
\begin{align*}
 F(x)=\frac{f(x)}{x}=\frac{1+x}{1-Ax-Cx^2} \mathcal{C}\!\left(\frac{x(1+x)(1+Bx+Dx^2)}{(1-Ax-Cx^2)^2}\right),
\end{align*}
and hence
\[
 F(x)=\frac{1+x}{1-Ax-Cx^2-x(1+Bx+Dx^2)F(x)}.
\]
Set \(p=A+2\), \(R(x)=1/F(x)\), and define \(Q(x)\) by
\begin{equation}\label{eq:amatrix-c14-jacobi}
 R(x)=1-px-x^2Q(x).
\end{equation}
By eliminating \(F(x)\) and \(R(x)\), we obtain
\begin{equation}\label{eq:amatrix-c14-Q}
 Q(x)=\frac{q+\ell x}{1-px+(C-2p)x^2+x^2(-1-x)Q(x)},
 \qquad q=B+C,\quad \ell=p(p-C)+D.
\end{equation}
By \eqref{eq:amatrix-c14-jacobi} and \cref{lem:jacobi-step}, we obtain 
\[H_{n+1}(F)=H_n(Q).
\]
Equation \eqref{eq:amatrix-c14-Q} is Equation \eqref{eq:wz-form} with
\[a_0=q,\quad b_0=\ell,\quad c_0=-p,\quad d_0=C-2p,\quad f_0=-1.
\]
Thus we have 
\[f_1=-\frac{\ell}{q},\qquad a_1=q-C+2p-\frac{p\ell}{q}-\frac{\ell^2}{q^2}.
\]
If
\[L=(p+1)q+\ell=4+A^2+3B+A(4+B)+C+D,
\]
then \(c_0+f_0+f_1=-L/q\).  Equations
\eqref{eq:wz-alpha}--\eqref{eq:wz-beta} now give \(\alpha=L^2\);
expanding the second one gives exactly the displayed polynomial \(\beta\).
The specialization principle in \cref{rem:specialization} covers \(q=0\).
\end{proof}

\begin{corollary}{\em (Conjectured in \cite[Section 4]{Barry2019Pseudo})}\label{thm:pseudo-tail}
Let $g(x)$ be the power series defined as follows:
\[
g(x)=\frac{1}{1 - ax - bx^2} \mathcal{C} \left( \frac{-x^2(b + cx)}{(1 - ax - bx^2)^2} \right).
\]
Then the Hankel transform of \(g(x)\) forms an $\bigl((ab+c)^2,\;b(ab+c)^2\bigr)$ Somos-4 sequence.
\end{corollary}
\begin{proof}
According to the generating function of $g(x)$, the functional equation of $g(x)$ is given by 
\begin{align*}
 g(x)=\frac1{1-ax-bx^2+x^2(b+cx)g(x)}.
\end{align*}
Set \(Q(x)=-bg(x)\). Then we have
\[Q(x)=\frac{-b}{1-ax-bx^2+x^2(-1-(c/b)x)Q(x)}.
\]
This satisfies the Wang--Zhang condition with
\[a_0=-b,\quad b_0=0,\quad c_0=-a,\quad d_0=-b,\quad f_0=-c/b,\quad f_1=0,\quad a_1=0.
\]
The formulas in \cref{thm:wz} immediately give $\alpha=(ab+c)^2$ and $\beta=b(ab+c)^2$.
The geometric scaling and an index shift do not change these parameters.
This completes the proof.
\end{proof}

The following assertion is even more directly a cubic continued fraction.

\begin{corollary}{\em (Conjectured in \cite[Section 6]{Barry2019Pseudo})}\label{thm:pseudo-asequence}
Let $F(x)$ be the power series defined as follows:
\[
F(x)=\frac{1}{1 + ax + bx^2} \mathcal{C} \left( \frac{x^3(ab + c)}{(1 + ax + bx^2)^2} \right).
\]
Then the Hankel transform \((H_{n+1}(F))_{n\geq 0}\) forms an $\bigl((ab+c)^2,\;b(ab+c)^2\bigr)$ Somos-4 sequence.
\end{corollary}
\begin{proof}
According to the generating function of $F(x)$, the continued fraction of $F(x)$ is given by 
\[
 F(x)=\frac{1}{1+ax+bx^2-(ab+c)x^3F(x)}.
\]
The result follows from \cref{thm:cubic-master} with $(r,s,t)=(-a,-b,ab+c)$.
\end{proof}

\section{Periodic Hankel evaluations}\label{Section-Seven-Hankel-Evalu}

\subsection{Periodic quadratic transformations}

The explicit evaluations in \cite[Conjectures~26, 30, and 39]{Barry2021Catalan} require periodic quadratic transformations rather than a single application of \cref{thm:wz}. They are proved in \cref{thm:2021-c26-c30,thm:2021-c39} below.

Set \(\kappa=2-\alpha\), and define
\[ d_{-1}=d_0=1,\qquad d_{m+1}=3d_m-\kappa d_{m-1}\quad(m\geq0).
\]
Thus we have
\begin{equation}\label{eq:2021-d-generating}
 \sum_{m\geq0}d_my^m=\frac{1-\kappa y}{1-3y+\kappa y^2}.
\end{equation}

\begin{theorem}{\em (Conjectured in \cite[Conjecture 26 and 30]{Barry2021Catalan})}\label{thm:2021-c26-c30}
Let $F(x)$ be the generating function given by
\begin{align*}
 F(x)=\frac{1-x-\alpha x^2}{1-x-x^2}\mathcal{C}\!\left(\frac{x^2(1-x-\alpha x^2)}{(1-x-x^2)^2}\right).
\end{align*}
For \(h_n=H_{n+1}(F)\), we have
\begin{equation}\label{eq:2021-c26-result}
 h_n=\kappa^{\lfloor(n+1)^2/4\rfloor}d_{\lfloor n/2\rfloor}.
\end{equation}
Equivalently,
\[h_n=(2-\alpha)^{\lfloor(n+1)^2/4\rfloor}[x^n]\frac{(1+x)(1+(\alpha-2)x^2)}{1-3x^2-(\alpha-2)x^4}.
\]
Conjecture~30 in \cite{Barry2021Catalan} is the same formula with \(\kappa=\beta\).
\end{theorem}
\begin{proof}
By the definition of $F(x)$, we obtain the quadratic equation
\begin{equation}\label{eq:2021-c26-quadratic}
x^2F(x)^2-(1-x-x^2)F(x)+1-x-\alpha x^2=0.
\end{equation}
For a parameter \(R\), let \(F_R(0)=1\) be the distinguished solution of
\begin{equation}\label{eq:2021-c26-tail}
 Rx^2F_R^2+\bigl(-1+x+(3-2R)x^2\bigr)F_R
 +1-x+\left(R-3+\frac{\kappa}{R}\right)x^2=0.
\end{equation}

Set \(u = \kappa/R\) and \(R' = 3 - u\). A direct substitution yields the continued-fraction decomposition
\begin{equation}\label{eq:2021-c26-two-steps}
 F_R=\frac1{1-ux^2G_R},\qquad
 G_R=\frac1{1-x-R'x^2F_{R'}}.
\end{equation}
To verify this, substitute the expression for $F_R$ into \eqref{eq:2021-c26-tail}; using the defining equation for $F_{R'}$ and the relations $u=\kappa/R$, $R'=3-u$, one obtains an identity.

This transformation is \emph{covariant}: it preserves the form of the quadratic equation, merely replacing the parameter \(R\) by \(R'\).

Starting from \(R_0 = 1\), repeated application of \eqref{eq:2021-c26-two-steps} gives the parameter iteration
\[
R_{m+1} = 3 - \frac{\kappa}{R_m}.
\]
We linearize this iteration by defining the sequence \((d_m)\) via
\[
d_{-1} = d_0 = 1, \qquad d_{m+1} = 3d_m - \kappa d_{m-1} \quad (m\ge 0).
\]
So we have
\[
d_{-1} = d_0 = 1, \qquad \frac{d_{m+1}}{d_m} = 3 - \frac{\kappa d_{m-1}}{d_m} =3 - \frac{\kappa} {\frac{d_m}{ d_{m-1}}}  \quad (m\ge 0).
\]
With the same initial values and recurrence, we get
\[
R_m = \frac{d_m}{d_{m-1}}.
\]

We work over the rational function field \(\mathbb{Q}(\kappa)\). For \(\kappa = 0\), one has \(d_m = 3^m\) for all \(m \geq 0\), 
so each \(d_m\) is a nonzero polynomial in \(\kappa\). Consequently, \(R_m = d_m / d_{m-1} \neq 0\) in \(\mathbb{Q}(\kappa)\), and all subsequent quotient operations are well defined.
For each fixed \(n\), both sides of the asserted determinant formula
belong to \(\mathbb Q[\kappa]\). Since they agree in \(\mathbb Q(\kappa)\), they agree as polynomials and therefore specialize
to every value of \(\kappa\), including values for which some \(d_m\)
vanishes.

Iterating the continued-fraction decomposition yields the Jacobi fraction
\[
F(x) = \cfrac{1}{1 - u_0 x - \cfrac{v_1 x^2}{1 - u_1 x - \cfrac{v_2 x^2}{1 - u_2 x - \cdots}}},
\]
whose coefficients are explicitly given by
\begin{equation}\label{eq:coeffs}
u_{2m}=0,\qquad u_{2m+1}=1,\qquad
v_{2m+1}=\kappa\frac{d_{m-1}}{d_m},\qquad
v_{2m+2}=\frac{d_{m+1}}{d_m}\quad(m\ge0).
\end{equation}
For a Jacobi continued fraction with coefficients \(v_j\), the classical formula states that
\begin{equation}\label{eq:prod}
H_{n+1}(F) = \prod_{j=1}^{n} v_j^{\,n+1-j}.
\end{equation}
Substituting the explicit values from \eqref{eq:coeffs} into \eqref{eq:prod}, we compute the product according to the parity of $n$.

\medskip\noindent\textbf{Case 1: $n=2m-1$ ($m\ge1$).}
Then
\[
H_{2m}(F)=\prod_{j=1}^{2m-1} v_j^{\,2m-j}.
\]
Separating odd and even indices and using \eqref{eq:coeffs}, we get
\[
H_{2m}(F)
=
\prod_{r=0}^{m-1} v_{2r+1}^{\,2m-2r-1}
\prod_{r=1}^{m-1} v_{2r}^{\,2m-2r}.
\]
Substituting $v_{2r+1}=\kappa d_{r-1}/d_r$ and $v_{2r}=d_r/d_{r-1}$, the product telescopes:
\begin{align*}
  H_{2m}(F) &=\kappa^{\sum_{r=0}^{m-1}(2m-2r-1)}
\cdot
\left(\frac{d_{m-2}}{d_{m-1}}\right) \cdot \left(\frac{d_{m-3}}{d_{m-2}}\right)^3 \cdots \left(\frac{d_{-1}}{d_{0}}\right)^{2m-1}\cdot \left(\frac{d_{m-1}}{d_{m-2}}\right)^2 \cdots \left(\frac{d_{1}}{d_{0}}\right)^{2m-2} \\
   & =\kappa^{m^2} d_{m-1}.
\end{align*}

\noindent\textbf{Case 2: $n=2m$ ($m\ge0$).}
Similarly,
\[H_{2m+1}(F)=\prod_{j=1}^{2m} v_j^{\,2m+1-j}=\prod_{r=0}^{m-1} v_{2r+1}^{\,2m-2r}
\prod_{r=1}^{m} v_{2r}^{\,2m+1-2r}=\kappa^{m(m+1)} d_m.\]

These two cases are unified by
\[H_{n+1}(F) = \kappa^{\lfloor (n+1)^2/4 \rfloor} d_{\lfloor n/2 \rfloor}.
\]
Finally, using the generating function
\[\sum_{m\ge 0} d_m y^m = \frac{1 - \kappa y}{1 - 3y + \kappa y^2},
\]
we need prove the equivalent coefficient form
\[
H_{n+1}(F)=(2-\alpha)^{\lfloor (n+1)^2/4 \rfloor}[x^n]\,\frac{(1+x)(1+(\alpha-2)x^2)}{1-3x^2-(\alpha-2)x^4}.
\]
Substituting \( y = x^2 \) into the generating function of \( d_m \) yields
\begin{equation}\label{eq:dm}
\sum_{m\ge 0} d_m x^{2m}=\frac{1 - \kappa x^2}{1 - 3x^2 + \kappa x^4}. 
\end{equation}
Multiplying both sides of \eqref{eq:dm} by \( 1 + x \), we obtain
\begin{equation}\label{eq:dmx}
(1+x) \sum_{m\ge 0} d_m x^{2m}=\sum_{m\ge 0} d_m x^{2m}+\sum_{m\ge 0} d_m x^{2m+1}.
\end{equation}
Extracting the coefficient of \( x^n \) in \eqref{eq:dmx}, we distinguish two cases:

\begin{itemize}
\item If \( n = 2m \) is even, the coefficient is \( d_m \);
\item If \( n = 2m + 1 \) is odd, the coefficient is also \( d_m \).
\end{itemize}
In both instances, the resulting coefficient equals \( d_{\lfloor n/2 \rfloor} \). Hence
\[
[x^n] (1+x) \frac{1 - \kappa x^2}{1 - 3x^2 + \kappa x^4}
= d_{\lfloor n/2 \rfloor}. 
\]
This is precisely the desired equivalent coefficient form.

Conjecture~30 in \cite{Barry2021Catalan} follows from \(\beta=2-\alpha=\kappa\).
This completes the proof.
\end{proof}

\begin{lemma}\label{lem:three-step-block}
Suppose
\[F(x)=\frac1{1-ux^3G(x)},\qquad G(x)=\frac1{1-x-x^2-vx^3E(x)}.
\]
Then, for \(N\geq3\), we have
\begin{equation}\label{eq:three-step-block}
 H_N(F)=-u^{N-1}v^{N-3}H_{N-3}(E).
\end{equation}
\end{lemma}
\begin{proof}
For a formal power series $A(x)$, define the associated bivariate Hankel kernel by
\[
\mathcal{K}_A(x,y)=\frac{xA(x)-yA(y)}{x-y}.
\]
Since \(F(x)=1/(1-ux^3G(x))\), we have \(1/F(x)=1-ux^3G(x)\). Applying  \cref{prop:unit_congruence}, 
the unitriangular congruence that multiplies $\mathcal K_F(x,y)$ on the left and on the right by $1/F(x)$ and $1/F(y)$, respectively, yields
\begin{align*}
 \frac{\mathcal{K}_F(x,y)}{F(x)F(y)} &=\frac{1}{F(x)F(y)}\cdot \frac{xF(x)-yF(y)}{x-y}=\frac{x/F(y)-y/F(x)}{x-y} \\
   & =\frac{x-y - ux^3yG(y) + uxy^3G(x)}{x-y}= 1 + uxy\,\frac{x^2G(x)-y^2G(y)}{x-y}.
\end{align*}
Since the congruence is unitriangular, it preserves every leading principal minor. Therefore, for each $N\ge 1$, 
the $N\times N$ leading principal block is transformed into
\[
1 \oplus \bigl( u \cdot H_{N-1}(xG) \bigr),
\]
hence, we can have \(H_N(F)=u^{N-1}H_{N-1}(xG)\).

It remains to evaluate $H_{N-1}(xG)$. To this end, set
\[
L_G(x,y):=\frac{x^2G(x)-y^2G(y)}{x-y}.
\]

Since \(1/G(x)=1-x-x^2-vx^3E(x)\), we apply  \cref{prop:unit_congruence} once more, 
dividing by $G(x)G(y)$ gives the congruent kernel
\begin{align*}
  \frac{L_G(x,y)}{G(x)G(y)} & =\frac{1}{G(x)G(y)}\cdot \frac{x^2G(x)-y^2G(y)}{x-y} =\frac{x^2/G(y)-y^2/G(x)}{x-y} \\
   & =\frac{x^2(1-y-y^2-vy^3E(y))-y^2(1-x-x^2-vx^3E(x))}{x-y} \\
   & =\frac{x^2-y^2 - xy(x-y) - vx^2y^2\bigl(yE(y)-xE(x)\bigr)}{x-y}\\
   &=x+y-xy + vx^2y^2\,\frac{xE(x)-yE(y)}{x-y}\\
   &=x+y-xy + vx^2y^2\mathcal{K}_E(x,y).
\end{align*}

The constant part $x+y-xy$ corresponds to the $2\times 2$ matrix
\[
M:=
\begin{pmatrix}
0 & 1 \\
1 & -1
\end{pmatrix},
\qquad \det M = -1.
\]
Hence, the unitriangular congruence transforms each leading principal block of $L_G$ into
\[
M \oplus \bigl( v \cdot H_{N-3}(E) \bigr)
\]
for every $N\ge 3$. Consequently,
\(H_{N-1}(xG)=\det (M)\cdot v^{N-3}H_{N-3}(E) =-v^{N-3}H_{N-3}(E).\)

Therefore, we obtain \(H_N(F)=(-1)\cdot u^{N-1}\cdot v^{N-3}H_{N-3}(E).\)
\end{proof}

\begin{theorem}{\em (Conjectured in \cite[Conjecture 39]{Barry2021Catalan})}\label{thm:2021-c39}
Let $F(x)$ be the generating function given by
\begin{equation}\label{eq:2021-c39-F}
 F(x)=\frac{1-x-x^2-\alpha x^3}{1-x-x^2-x^3} \mathcal{C}\!\left(\frac{x^3(1-x-x^2-\alpha x^3)}{(1-x-x^2-x^3)^2}\right),
\end{equation}
set \(h_n=H_{n+1}(F)\), and retain
\( \kappa=2-\alpha\) and \(d_m\) from \eqref{eq:2021-d-generating}.  Then we have
\begin{align}
 h_{3m}&=(-1)^m\kappa^{3m(m+1)/2}d_m,\label{eq:2021-c39-h0}\\
 h_{3m+1}&=0,\label{eq:2021-c39-h1}\\
 h_{3m+2}&=(-1)^{m+1}
 \kappa^{(m+1)(3m+4)/2}d_m,\label{eq:2021-c39-h2}
\end{align}
for all integers $m\ge 0$.

Equivalently, there exist sequences $A_n(\alpha)$ and $B_n$ such that
\[
h_n = A_n(\alpha)\, \kappa^{\,B_n},
\]
where
\[
A_n(\alpha)=[x^n]\,
\frac{1+(\alpha-2)x^2-(\alpha-2)x^3+(4\alpha-5)x^5-(\alpha-1)(\alpha-2)x^8}
{1+3x^3-(\alpha-2)x^6},
\]
and
\[
B_n=[x^n]\,
\frac{x(1-x+2x^2-2x^3+3x^4-3x^5+x^6)}
{(1-x)^2(1-x^3)}.
\]
This is precisely the Hankel determinant result in the coefficient-exponent form conjectured in \cite[Conjecture~39]{Barry2021Catalan}.
\end{theorem}
\begin{proof}
This proof is similar to that of \cref{thm:2021-c26-c30}.
By the definition of \(F(x)\), we obtain the quadratic equation
\begin{equation}\label{eq:2021-c39-quadratic}
 x^3F(x)^2-(1-x-x^2-x^3)F(x)+1-x-x^2-\alpha x^3=0.
\end{equation}
For a parameter \(R\), let \(F_R(0)=1\) be the distinguished solution of
\begin{align}\label{eq:2021-c39-tail}
 Rx^3F_R^2+\bigl(-1+x+x^2+(3-2R)x^3\bigr)F_R+1-x-x^2+\left(R-3+\frac{\kappa}{R}\right)x^3=0,
\end{align}
where $\kappa=2-\alpha$. 

Set \(u=\kappa/R\) and \(R'=3-u\).  Direct elimination gives the
closed pair of quadratic transformations
\begin{align*}
 F_R=\frac1{1-ux^3G_R},\qquad
 G_R=\frac1{1-x-x^2-R'x^3F_{R'}}.
\end{align*}
Substitution of the coefficients in \eqref{eq:2021-c39-tail} yields the same family with \(R\) replaced by \(R'=3-\kappa/R\).

Now set $F=F_1$ in \eqref{eq:2021-c39-quadratic}. Define sequences $(R_m)_{m\ge 0}$ and $(u_m)_{m\ge 0}$ by
\[
 R_m=\frac{d_m}{d_{m-1}},\qquad
 u_m=\frac{\kappa}{R_m}=\kappa\frac{d_{m-1}}{d_m},\qquad
 R_{m+1}=3-\frac{\kappa}{R_m}.
\]
Iterating the reduction formula from \cref{lem:three-step-block} gives 
\begin{equation}\label{eq:iteration-result}
H_N(F_1)=(-1)^m
\prod_{j=0}^{m-1}
\left(u_j^{N-3j-1}R_{j+1}^{N-3j-3}\right)
H_{N-3m}(F_{R_m}),
\end{equation}
whenever \(N-3m\in\{0,1,2\}\).  

Substituting the explicit expressions
\(u_j=\kappa d_{j-1}/d_j\) and
\(R_{j+1}=d_{j+1}/d_j\) into \eqref{eq:iteration-result}, all intermediate $d_j$'s telescope. 
Using the initial values $H_0(F_R)=H_1(F_R)=1$ and $H_2(F_R)=0$, we obtain
\begin{align*}
 H_{3m}(F_1)=(-1)^m\kappa^{m(3m+1)/2}d_{m-1},\quad
 H_{3m+1}(F_1)=(-1)^m\kappa^{3m(m+1)/2}d_m,\quad
 H_{3m+2}(F_1)=0.
\end{align*}
Replacing \(N\) by \(n+1\) proves \eqref{eq:2021-c39-h0} to \eqref{eq:2021-c39-h2}.

It remains to prove Barry's conjecture concerning the relationship between the coefficient sequences and the Hankel determinants.
Define
\[\mathcal A(x)=\frac{1+(\alpha-2)x^2-(\alpha-2)x^3+(4\alpha-5)x^5-(\alpha-1)(\alpha-2)x^8}{1+3x^3-(\alpha-2)x^6}.
\]
Since $\kappa=2-\alpha$, this becomes
\[\mathcal A(x)=\frac{1-\kappa x^2+\kappa x^3+(3-4\kappa)x^5+\kappa(1-\kappa)x^8}{1+3x^3+\kappa x^6}.
\]
Substituting $y=-x^3$ into
\[\sum_{m\geq0}d_my^m=\frac{1-\kappa y}{1-3y+\kappa y^2}
\]
gives
\[\sum_{m\geq0}(-1)^m d_mx^{3m}=\frac{1+\kappa x^3}{1+3x^3+\kappa x^6}.
\]
A direct calculation shows that
\[\mathcal A(x)=(1-x^2)\frac{1+\kappa x^3}{1+3x^3+\kappa x^6}+(1-\kappa)x^2.
\]
Therefore
\begin{align*}
\mathcal A(x)
&=(1-x^2)\sum_{m\geq0}(-1)^md_mx^{3m}+(1-\kappa)x^2\\
&=\sum_{m\geq0}(-1)^md_mx^{3m}+\sum_{m\geq0}(-1)^{m+1}d_mx^{3m+2}+(1-\kappa)x^2.
\end{align*}
Thus $A_n(\alpha)=[x^n]\mathcal A(x)$ satisfies
\begin{align*}
A_{3m}=(-1)^md_m,\qquad A_{3m+1}=0, \qquad A_{3m+2}=(-1)^{m+1}d_m,\qquad(m\geq1).
\end{align*}
At $m=0$, the additional term $(1-\kappa)x^2$ gives
\[A_2=-d_0+(1-\kappa)=-1+1-\kappa=-\kappa.
\]
Hence
\begin{equation}\label{eq:A-formula}
 A_n(\alpha)=
\begin{cases}
(-1)^m d_m,&n=3m,\\[2mm]
0,&n=3m+1,\\[2mm]
-\kappa,&n=2,\\[2mm]
(-1)^{m+1}d_m,&n=3m+2,\quad m\geq1.
\end{cases}
\end{equation}

Next, let
\[
\mathcal B(x)=
\frac{x(1-x+2x^2-2x^3+3x^4-3x^5+x^6)}
{(1-x)^2(1-x^3)}.
\]
The following partial-fraction decomposition is immediate:
\[\mathcal B(x)=x-x^2+\frac{1}{3(1-x)^3}-\frac{4}{9(1-x)}+\frac{1-5x+4x^2}{9(1-x^3)}.
\]
Using
\begin{align*}
\frac1{(1-x)^3}=\sum_{n\geq0}\binom{n+2}{2}x^n\quad \text{and}\quad \frac{1-5x+4x^2}{1-x^3}=\sum_{m\geq0}\bigl(x^{3m}-5x^{3m+1}+4x^{3m+2}\bigr),
\end{align*}
we can obtain the coefficients $B_n=[x^n]\mathcal B(x)$ according to the residue class of $n$ modulo $3$.
For $n=3m$, there is no contribution from $x-x^2$, and therefore
\begin{align*}
B_{3m}=\frac13\binom{3m+2}{2}-\frac49+\frac19=\frac{3m(m+1)}2.
\end{align*}
For $n=3m+2$, we obtain
\[B_{3m+2}=\frac13\binom{3m+4}{2}-\mathbf 1_{\{m=0\}},
\]
because the term $-x^2$ contributes only when $m=0$. Hence $B_2=1$, whereas, for $m\geq1$,
\begin{align*}
B_{3m+2}=\frac{(3m+3)(3m+4)}{6}=\frac{(m+1)(3m+4)}2.
\end{align*}
For completeness, the remaining residue class is
\[B_1=1,\qquad B_{3m+1}=\frac{m(3m+5)}2\quad(m\geq1),
\]
although these values will not affect $h_{3m+1}$, since $A_{3m+1}=0$.

We now verify $h_n=A_n(\alpha)\kappa^{B_n}$ for all $n\ge 0$.
If $n=3m$, then by \eqref{eq:A-formula} and the formula for $B_{3m}$,
\[A_{3m}(\alpha)\kappa^{B_{3m}}=(-1)^md_m\kappa^{3m(m+1)/2}=h_{3m}.
\]
If $n=3m+1$, then
\[A_{3m+1}(\alpha)\kappa^{B_{3m+1}}=0=h_{3m+1}.
\]
If $n=3m+2$ with $m\geq1$, then
\[A_{3m+2}(\alpha)\kappa^{B_{3m+2}}=(-1)^{m+1}d_m\kappa^{(m+1)(3m+4)/2}=h_{3m+2}.
\]
Finally, when $n=2$, namely $m=0$, we have $A_2=-\kappa$, $B_2=1$, and therefore $A_2\kappa^{B_2}=-\kappa^2$.
Since $d_0=1$, we have $h_2=(-1)^{1}\kappa^{(1)(4)/2}d_0=-\kappa^2$.
Thus the identity also holds at $n=2$.

Combining all cases, we conclude that for every $n\ge 0$, \(h_n=A_n(\alpha)\kappa^{B_n}\).
This completes the proof.
\end{proof}

\subsection{The OEIS pair A136576--A136577}\label{sec:new-a136576}

Let \(u_n\) be [A136576] on OEIS \cite{Sloane23}, and let 
\[ U(x)=\sum_{n\geq0}u_nx^n,\qquad F(x)=\frac{U(x)}x=\sum_{n\geq0}u_{n+1}x^n.
\]
The generating function of $U(x)$ is given by 
\begin{align}\label{Equat-UX-OEIS}
U(x)=\frac{\sqrt{1+4x-4x^2}+4x^2-2x-1}{8x^2}.
\end{align}

\begin{theorem}{\em (Conjectured in \cite[A136576; A136577]{Sloane23})}\label{thm:new-a136576}
The Hankel transform of \((u_{n+1})_{n\geq0}\) is the OEIS sequence [A136577].  If \(h_n=H_{n+1}(F)\), then we have 
\begin{align}
h_{4m}&=(-1)^m4^{4m^2},\qquad  h_{4m+1}=(-1)^m4^{4m^2+2m}, \label{eq:new-a136577-0}\\
h_{4m+2}&=0,\qquad\qquad h_{4m+3}=(-1)^{m+1}4^{4m^2+6m+2},\label{eq:new-a136577-2}
\end{align}
for \(m\geq0\). Thus
\[(h_n)_{n\geq0}=1,1,0,-16,-256,-4096,0,16777216,\ldots.
\]
\end{theorem}
\begin{proof}
By Equation \eqref{Equat-UX-OEIS} and \(F(0)=1\), we have
\begin{align*}
 F(x)=\frac{1-x}{1+2x-4x^2+4x^3F(x)}.
\end{align*}
Apply \cref{lem:sx} with $(a,b,c,d,e,f)=(1,-1,2,-4,0,4)$.
Then \(H_N(F)=H_{N-1}(G)\), where
\begin{equation}\label{eq:new-a136576-G}
 G(x)=\frac{1-x}{1+2x-2x^2+x^2(-1+x)G(x)}.
\end{equation}
This satisfies the Wang--Zhang condition (Equation \eqref{eq:wz-form}) with 
\[(a_0,b_0,c_0,d_0,f_0)=(1,-1,2,-2,1),\qquad f_1=1,\qquad a_1=0.
\]
Thus \(\bigl(H_n(G)\bigr)\) is a \((16,-16)\) Somos--4 sequence.
To identify it exactly, however, we need one more determinant reduction.

Applying \cref{lem:sx} once more, now to \eqref{eq:new-a136576-G}, gives, for \(n\geq1\),
\[ H_n(G)=H_{n-1}(Q),\qquad Q(x)=\frac{-4x}{1+2x-4x^2+x^2(-1+x)Q(x)}.
\]
Write \(Q(x)=-4xL(x)\) and \(E(x)=(1-x)L(x)\). 
Therefore, we have $Q(x)=-\frac{4xE(x)}{1-x}$.
We obtain
\[-\frac{4xE(x)}{1-x}=\frac{-4x}{1+2x-4x^2+x^2(-1+x)\left(-\frac{4xE(x)}{1-x}\right)}.
\]
Since $x^2(-1+x)(-\frac{4xE(x)}{1-x})=4x^3E(x)$, we obtain $(1+2x-4x^2)E(x)+4x^3E(x)^2=1-x$. 
Now define
\[G_0(x)=\frac{1-4xE(x)}{1-x}.
\]
We verify that $G_0$ satisfies \eqref{eq:new-a136576-G}. The denominator on the right-hand side of \eqref{eq:new-a136576-G} becomes
\[1+2x-2x^2+x^2(-1+x)G_0=1+2x-3x^2+4x^3E(x).
\]
To show $G_0$ satisfies \eqref{eq:new-a136576-G}, it suffices to prove
\[
(1-4xE(x))(1+2x-3x^2+4x^3E(x))=(1-x)^2. 
\]
Since 
\[1-x = (1+2x-4x^2)E(x)+4x^3E(x)^2,
\]
we obtain
\[
(1-4xE(x))(1+2x-3x^2+4x^3E(x))-(1-x)^2=4x((1-x)-(1+2x-4x^2)E(x)-4x^3E(x)^2)=0.
\]
Since the solution of \eqref{eq:new-a136576-G} as a formal power series is unique, we conclude
\[
G(x)=G_0(x)=\frac{1-4xE(x)}{1-x}.
\]

Now we derive the expression for $E(x)$ in terms of $G(x)$.
We have $4xE(x) = 1-(1-x)G(x)$.
On the other hand, rewrite \eqref{eq:new-a136576-G} as
\[(1+2x-2x^2)G - x^2(1-x)G^2 = 1-x.
\]
We obtain
\begin{align*}
\bigl(1-(1-x)G\bigr)(1+3x-x^2G)&=1+3x - (1+2x-2x^2)G + x^2(1-x)G^2
\\ &=1+3x-(1-x)=4x.
\end{align*}
Therefore, we obtain 
\begin{align}\label{eq:new-a136576-E}
E(x)=\frac{1-(1-x)G(x)}{4x}=\frac{1}{1+3x-x^2G(x)}.
\end{align}

For a series \(A(x)\), use the Hankel kernel
\[\mathcal K_A(x,y)=\frac{xA(x)-yA(y)}{x-y}.
\]
A direct substitution of the equation for \(L(x)\) gives
\[\frac{\mathcal K_{xL}(x,y)}{L(x)L(y)}=x+y+2xy-4x^2y^2\mathcal K_E(x,y).
\]
Multiplication on the two sides by the unit series \(L(x)^{-1}\) and \(L(y)^{-1}\) is a unitriangular congruence.  The coefficient matrix on the right is the direct sum of
\[\begin{pmatrix}0&1\\1&2\end{pmatrix}
\quad\text{and}\quad
 -4\ \text{times the Hankel matrix of }E.
\]
Consequently, for \(m\geq2\),
\begin{equation}\label{eq:new-a136576-block}
 H_m(xL)=-(-4)^{m-2}H_{m-2}(E).
\end{equation}
Using \(Q(x)=-4xL(x)\), \eqref{eq:new-a136576-block}, and \cref{lem:jacobi-step} in \eqref{eq:new-a136576-E}, we obtain
\[ H_n(G)=H_{n-1}(Q)=(-4)^{n-1}H_{n-1}(xL)=-16^{\,n-2}H_{n-3}(E)=-16^{\,n-2}H_{n-4}(G)\]
for \(n\geq4\).  Finally,
\[
 H_0(G)=1,\qquad H_1(G)=1,\qquad
 H_2(G)=0,\qquad H_3(G)=-16.
\]
Iteration of the four-step identity yields
\eqref{eq:new-a136577-0}--\eqref{eq:new-a136577-2}.  Since
\(h_n=H_n(G)\), these are exactly the terms of [A136577] on OEIS \cite{Sloane23}.
\end{proof}

\begin{remark}\label{rem:new-a136577}
Note that the formula of [A136577] on OEIS is given by
$$a(n)=4^{\left\lfloor \frac{n^2}{2}\right\rfloor} \left(\left(\frac{1}{2}-\frac{\sqrt{2}}{2}\right)\cos\left(\frac{3\pi n}{4}\right)+\left(\frac{1}{2}+\frac{\sqrt{2}}{2}\right)\cos\left(\frac{\pi n}{4}\right)\right).$$
The terms and \cref{thm:new-a136576} show that the exponent \(\lfloor n^2/2\rfloor\) currently printed in the closed formula for [A136577] must be \(\lfloor n^2/4\rfloor\). The former already gives \(-256\), rather than \(-16\), at \(n=3\).
Therefore, we have
$$a(n)=4^{\left\lfloor \frac{n^2}{4}\right\rfloor} \left(\left(\frac{1}{2}-\frac{\sqrt{2}}{2}\right)\cos\left(\frac{3\pi n}{4}\right)+\left(\frac{1}{2}+\frac{\sqrt{2}}{2}\right)\cos\left(\frac{\pi n}{4}\right)\right).$$
\end{remark}

\section{By-products and concluding remarks}\label{sec:new-motzkin}

Next we prove a byproduct of this paper, namely a conjecture concerning generating functions.

Let \(R_{n,k}\) be the left-justified Pascal rhombus, whose bivariate generating function is
\begin{equation}\label{eq:new-pascal-rhombus}
 P(x,z)=\sum_{n,k\geq0}R_{n,k}x^nz^k=\frac1{1-x-xz-xz^2-x^2z^2}.
\end{equation}
Write
\[ E(x,y)=\sum_{n,k\geq0}R_{n,2k}x^ny^k
\]
for the even-column bisection.
The assertion in \cite[Conjecture~12]{Barry2023Motzkin} contains a sign error as printed; we prove the corrected statement.

\begin{proposition}{\em (Conjectured in \cite[Conjecture~12]{Barry2023Motzkin})}\label{thm:new-conjecture-12}
The generating function of the bisection \((R_{n,2k})\) is
\begin{equation}\label{eq:new-conjecture-12-corrected}
 E(x,y)=\frac{1-(y+1)x-yx^2}{1-2(y+1)x+(1-y+y^2)x^2+2y(y+1)x^3+y^2x^4}.
\end{equation}
\end{proposition}
\begin{proof}
For this identity, apply the even-part filter to \eqref{eq:new-pascal-rhombus}. Set $A=1-(1+z^2)x-z^2x^2$.
Then
\[ P(x,z)=\frac1{A-xz},\qquad P(x,-z)=\frac1{A+xz},
\]
and hence
\[ E(x,z^2)=\frac{P(x,z)+P(x,-z)}2=\frac{A}{A^2-x^2z^2}.
\]
On setting \(y=z^2\), the numerator is \(1-(1+y)x-yx^2\), while
\[A^2-x^2z^2=1-2(1+y)x+(1-y+y^2)x^2+2y(1+y)x^3+y^2x^4.
\]
This proves \eqref{eq:new-conjecture-12-corrected}.
\end{proof}

\begin{corollary}\label{cor:new-conjecture-12-sums}
Define the row and diagonal sums by
\[r_n=\sum_{k=0}^{n}R_{n,2k},\qquad e_m=\sum_{k=0}^{\lfloor m/2\rfloor}R_{m-k,2k}.
\]
Then we have generating functions
\begin{align*}
\sum_{n\geq0}r_nx^n=E(x,1) &=\frac{1-2x-x^2}{1-4x+x^2+4x^3+x^4},\\
\sum_{m\geq0}e_mx^m=E(x,x)&=\frac{1-x-x^2-x^3}{(1+x+x^2)(1-3x+x^2+x^3+x^4)}.
\end{align*}
\end{corollary}
\begin{proof}
These are the substitutions \(y=1\) and \(y=x\) in \eqref{eq:new-conjecture-12-corrected}.
\end{proof}

To the best of our knowledge, we have proved almost all of the currently known conjectures related to both the $(\alpha,\beta)$ Somos-4 sequences and Hankel determinants. The resolution of these conjectures can deepen the reader's understanding of Somos-$k$ sequences. In particular, it establishes a solid groundwork for attacking \cref{Openproblem}.

This paper can also serve as a nice application of the Sulanke--Xin quadratic transformation (see \cref{lem:sx}; also see \cite{SulankeXin2008}) and the Wang--Zhang condition (see \cref{thm:wz}). Further applications are worth discovering.






\noindent
{\small \textbf{Acknowledgments:}}
The authors would like to express sincere gratitude for all the suggestions that have improved the presentation of this paper.
The authors thank Professor Guoce Xin for bringing to our attention the conjecture concerning sequence [A136577] on OEIS.
Feihu Liu was partially supported by the Postdoctoral Fellowship Program and China Postdoctoral Science Foundation (Grant No. BX2026002).
Ying Wang was partially supported by the Natural Science Foundation of Henan Province (Grant No. [262300422649]).

\end{document}